\documentclass{article}

\usepackage{ amsfonts, amssymb,amsmath}

\newcommand{\al}{\alpha}

\newcommand{\lm}{\lambda}
\newcommand{\vph}{\varphi}

\newcommand{\T}{\operatorname{T\,}}

\newcommand{\vrai}{\operatorname{ess\,}}
\newcommand{\wt}{\widetilde}

\usepackage{amsthm}

\newtheorem{theorem}{Theorem}[section]
\newtheorem{definition}[theorem]{Definition}
\newtheorem{remark}[theorem]{Remark}
\newtheorem{lemma}[theorem]{Lemma}
\newtheorem{corollary}[theorem]{Corollary}

\begin{document}
	\title{Well-posedness and robust stability of a nonlinear ODE-PDE system \thanks{Tis work was supported by the German Research
			Foundation (DFG) (Grant no. DA 767/12-1) and the
			National Research Foundation of Ukraine (NRFU) (Grant
			no. NRFU F81/41743) through the joint German-Ukrai-
			nian project ''Stability and robustness of attractors of
			nonlinear infinite-dimensional systems with respect to
			disturbances.''}}
	\author{S. Dashkovskiy \\
		University of W{\"u}rzburg \\
		\and 
		O. Kapustyan \\
		T. Shevchenko University of Kyiv \\
		\and V. Slynko\\
		University of W{\"u}rzburg
	}
\maketitle

% REQUIRED
\begin{abstract}
This work studies stability and robustness of a nonlinear system given as an interconnection of an ODE and a parabolic PDE subjected to external disturbances entering through the boundary conditions of the parabolic equation. To this end we develop an approach for a construction of a suitable coercive Lyapunov function as one of the main results. Based on this Lyapunov function we establish the well-posedness of the considered system and establish conditions that guarantee the ISS property.
ISS estimates are derived explicitly for the particular case of globally Lipschitz nonlinearities.
\end{abstract}

%	\item Research highlight 1
%	\item Research highlight 2

	
	

\section{Introduction}
%% What is the problem?
%% What is known?
%% What am I doing?
%% How am I doing it?
%% Organization of the article.

Studying stability of infinite dimensional systems has a long history of several decades \cite{Dat72},\cite{KrD74},\cite{Hen81}. During the last decade a lot of attention was devoted to the investigation of robust stability of such systems, especially in the input-to-state stability (ISS) framework
\cite{KaK19},\cite{MiW19},\cite{MiP20},\cite{JSZ19}. This framework is known to be suitable in studying interconnected systems. Coupled systems appear in many modern practical problems. A special class of interconnections are couplings of a PDE and ODE systems motivated by different applications from mechanics \cite{DaN07}, control \cite{GGT20},\cite{laurent},\cite{salvidar}, biology \cite{efendiev}, etc.
Since we are interested in stability properties, we recall several recent works where this property was studied for the case of ODE-PDE couplings.

The work \cite{salvidar} deals with the problem of axial and rotational vibrations in the drilling process, where the rock-bit interaction leads to a nonlinear coupling function between an ODE and the wave equation. A feedback stabilization controller was developed there to guarantee the ultimate boundedness of solutions so that the undesired vibrations are suppressed successfully. To this end the authors proposed a Lyapunov-Krasovskii functional and developed stability conditions in form of linear and bilinear matrix inequalities.

Theorems of the small-gain type were used in \cite{ahmed-ali} to establish exponential stability of a hybrid system given as an interconnection of an autonomous ODE and a parabolic PDE. Based on the ISS-type estimates the authors establish exponential stability for the case when a discrete time controller is applied. Their stability conditions restrict the maximal discretization of the time interval. By means of a coordinate transformations from \cite{krst} it was demonstrated that these conditions can be used for the design of suitable observers with discrete time output, for linear ODE-PDE cascades.

A feedback controller for local stabilization of a reaction-diffusion system was proposed in \cite{miron}. The Lyapunov stability of the uncontrolled system was not assumed there. By means of the direct Lyapunov method the authors derive estimates for the domain of attraction for the closed loop system, in the form of linear and bilinear inequalities. These results can be applied to the cases of distributed as well as scalar boundary controllers.

Boundary stabilization of a cascade of an ODE with a heat equation was considered in \cite{kang}. By means of the backstepping techniques the authors have designed a suitable stabilizing controller and derived estimates for the domain of attraction. These estimates were obtained by the direct Lyapunov method and with help of the Halanay inequality. These results were then extended in \cite{kang2} to the case of coupled linear ODE-PDE system with time varying delay.

Stabilization of a linear ODE-PDE system with boundary control was considered in \cite{yuan}. Using the backstepping approach the system was transformed into a cascade form for which a stabilizing controller was designed. By the direct  Lyapunov method exponential estimates for the norm of solutions of the original system were derived.

The authors of  \cite{tanwani} consider a system of linear hyperbolic equations, where the state of one boundary point is controlled by the state measured at another boundary point. Since the measurements are assumed to be perturbed the problem is to design controllers robust with respect to the measurement disturbances. For locally essentially bounded disturbances the conditions to guarantee the ISS or ISpS property are derived. The well-posedness and stability analysis are based on the $C_0$-semigroup theory and Lyapunov methods.

Adaptive stabilization problem for a cascade of an ODE with a hyperbolic PDE subjected to unknown harmonic disturbances was considered in \cite{Xu}. Stabilizing controllers were developed there. By means of the linear semigroup theory and La Salle invariance principle conditions of well-posedness and for the asymptotic stability for the closed loop system were established.

Some older related works are \cite{bartyshev, kedyk,kedyk-ob}, where stability of coupled parabolic PDE with an ODE was considered. These works use vector Lyapunov functions combined with theory of monotone dynamical systems in Banach spaces \cite{hirsh}.
Matrix valued Lyapunov functions were used in \cite{mart-sl} for stability investigation of ODE-PDE systems. The ISS framework for such and other infinite dimensional systems was used in \cite{KaK19}. 

Let us note that in the most of the above literature it is required that decoupled systems are asymptotically stable. This excludes a class of interconnections where one of the subsystems can be unstable, but the overall system is stabilized by the other one. Our result aims to fill this gap.

In this work we consider a coupled ODE with a parabolic PDE so that the decoupled ODE subsystem is not necessarily ISS. The disturbances enter to the system at the boundary, which is more difficult to handle than the distributed ones. We will derive conditions guaranteeing the ISS property for the coupled system under the assumption  that the linearized PDE is globally asymptotically stable. 

In contrary to well-known approaches such as vector Lyapunov functions or small-gain theory a construction of a Lyapunov function for the whole system in our case cannot be derived from the Lyapunov functions of subsystems, due to the presence of an unstable subsystem. Also the admissibility approach as used in \cite{JSZ19}, cannot be applied directly to our case because of the presence of essentially nonlinear functions in the subsystems. Even if we exclude such nonlinearities in our system this approach is hardly possible to extend as even in this case we will have difficulties to derive an explicit expression for the linear semigroup generated by the linear part of the system. 
Expansion of solution in series of eigenfunctions (of the linear part) as used for example in \cite{KaK19} cannot be applied in our case due to the essential nonlinearities. 

A new method is needed that takes nonlinearities of the subsystems, possible instability of one of them and the presence of disturbances at the boundary of the parabolic PDE into account.

To solve this problem we provide an approach to construct a Lyapunov function, which leads to the resolving of a boundary value problem for a second order PDE. The Green function will be derived for the latter problem explicitly, so that the needed Lyapunov function is obtained also explicitly. Based on this Lyapunov function we derive conditions for the ISS property of the coupled system. For the case of globally Lipschitz nonlinearities we provide an ISS-type estimation of solutions.

This paper is organized as follows. The next section introduces notation and several known basic facts that will be used in the main part of the paper. The problem statement is given in Section 3 together with related definitions. Several steps needed for the construction of a suitable Lyapunov function are explained in Section 4. The well-posedness and the main result establishing the ISS property is given in Section 5.  Concluding remarks are collected in Section 6. Technical proofs are provided in the Appendix.

\section{Notation and known facts}
By $C[0,l]$ we denote the space of continuous functions defined on $[0,l]$ with values in $\mathbb{R}$ normed by $\|f\|_{C[0,l]}=\max\limits_{x\in [0,l]}|f(x)|$, and  $C^k[0,l]$ for $k\in\mathbb{N}$ denotes the space of continuously differentiable up to the order $k$ functions with the norm
$\|f\|_{C^k[0,l]}=\max\limits_{p=0,\dots,k}\max\limits_{x\in [0,l]}|f^{(p)}(x)|$. Let
$L^p[0,l]$, $1\le p<\infty$ be the space of Lebesgue measurable functions with finite norm given by  $\|f\|_{L^p[0,l]}=(\int\limits_0^l|f(z)|^p\,dz)^{1/p}$. The Hilbert space $H^k[0,l]\subset L^2[0,l]$ is a subset of $L^2[0,l]$ of functions $f$ with (generalized) derivatives $f^{(p)}\in L^2[0,l]$ and scalar product is defined by
\begin{equation}\label{0.1}
(f,g)_{H^k[0,l]}=\sum\limits_{p=0}^k\int\limits_0^lf^{(p)}(z)g^{(p)}(z)\,dz.
\end{equation}
Recall that $H^k[0,l]$ is a completion of $C^k[0,l]$ with respect to the norm $\|f\|_{H^k[0,l]}=\sqrt{(f,f)_{H^k[0,l]}}$.
$C_0^{\infty}[0,l]$ denotes the set of infinitely differentiable functions  $[0,l]$ vanishing in a vicinity of the points $x=0$ and $x=l$. $C_0^{\infty}([0,T],C_0^{\infty}[0,l])$ denotes the set infinitely smooth mappings  $f\,:\,[0,T]\to C_0^{\infty}[0,l]$ with zero value in a vicinity of the points $t=0$ and $t=T$. The completion of  $C_0^{\infty}[0,l]$ with respect to the norm $H^k[0,l]$  is denoted by $H_0^k[0,l]$.
$H^{-k}[0,l]$ denotes the dual space to $H^k_0[0,l]$ with the standard norm. For $\alpha,\beta\in\Bbb Z_+$ by $H^{\alpha,\beta}([0,l]\times[0,T])$ we denote the space of such functions $f:[0,l]\times[0,T]\to\Bbb R$ such that $\partial^k_z f\in L^2([0,l]\times[0,T])$ for $k=0,1,\dots,\alpha$ and $\partial^k_t f\in L^2([0,l]\times[0,T])$ for $k=0,1,\dots,\beta$, see, e.g., \cite{Mik78}.

For a Banach space $\mathcal X$ its dual space is denoted by $\mathcal X^*$ and $\langle\cdot,\cdot\rangle:\mathcal X^*\times\mathcal X\to \mathbb R$ is the usual duality pairing.

Let $\mathcal X$ be a Banach space, then $L^p([0,T],\mathcal X)$, $1\le p<\infty$ denotes the space of strongly measurable mappings $[0,T]\to\mathcal X$ such that the norm
 $\|f\|_{L^p([0,T],\mathcal X)}=(\int\limits_0^T\|f(s)\|_{\mathcal X}^p\,ds)^{1/p}$ is finite. For $p=\infty$ this norm is defined by
$\|f\|_{L^{\infty}([0,T],\mathcal X)}=\vrai\sup_{t\in[0,T]}\|f(t)\|_{\mathcal X}$.

$L^{\infty}(\mathbb{R}_+)$ is the Banach space of measurable functions essentially bounded on $\mathbb{R}_+:=[0,\infty)$.
For $x\in\mathbb{R}^n$ we use the norm $\|x\|=(\sum_{i=1}^nx_i^2)^\frac12$.
$C^1(U,\mathbb{R}^n)$, $U\subset\mathbb{R}$ is the set of continuously differential mappings  $U\to\mathbb{R}^n$.

For $f\,:\,[0,l]\to (L^p[0,l])^n$, $1\le p<\infty$ the norm is defined by  $\|f\|_{L^p[0,l]}=(\sum\limits_{k=1}^n\|f_k\|_{L^p[0,l]}^p)^{1/p}$
and for $p=\infty$ it is  $\|f\|_{L^{\infty}[0,l]}=\max_{k=1,\dots,n}\|f_k\|_{L^{\infty}[0,l]}$.
We will use the following sets of comparison functions:
\begin{equation*}
	\mathcal K=\{\gamma\,:\,\mathbb{R}_+\to\mathbb{R}_+\,:\,\text{continuous, strictly increasing and } \,\,\gamma(0)=0\},
\end{equation*}	
\begin{equation*}
	\mathcal L=\{\gamma\,:\,\mathbb{R}_+\to\mathbb{R}_+\,:\,\text{continuous, strictly decreasing and}\,\, \lim\limits_{t\to\infty}\gamma(t)=0\},
\end{equation*}	
\begin{equation*}	
		\mathcal K\mathcal K\mathcal L=\{\beta\,:\,\mathbb{R}_+^2\times\mathbb{R}_+\to\mathbb{R}_+\,:\,\text{cont.,}\,\, \beta(\cdot,r,t),\beta(r,\cdot,t)\in\mathcal K,\,\,\beta(r,s,\cdot)\in\mathcal L\}.
\end{equation*}
$\mathbb{R}^{n\times m}$ is the linear space of
$n\times m$-matrices, and if $m=n$, then $\mathbb{R}^{n\times n}$ is a Banach algebra. 
By $\mathbb S^n$ we denote the set of the symmetric matrices of the size $n$. For $P,Q\in\mathbb S^n$, we write $P\succ Q$ if and only if $P-Q$ is positive definite.
For $A\in\mathbb S^{n}$ its minimal and maximal eigenvalues are denoted by $\lambda_{\min}(A)$ and  $\lambda_{\max}(A)$ respectively. The norm on $\mathbb{R}^{n\times n}$ is induced by the euclidean norm in
$\mathbb{R}^n$: $\|A\|=\sup_{\|x\|=1}\|Ax\|=\lambda_{\max}^{1/2}(A^{\T}A)$.

We will use the following well-known inequalities:
Young's inequality
\begin{equation}\label{Yang}
\gathered
xy\le \frac{x^{p_1}}{p_1}+\frac{y^{p_2}}{p_2},\quad x\ge 0,\quad y\ge 0,\quad p_1\in(1,\infty),\quad \frac{1}{p_1}+\frac{1}{p_2}=1,
\endgathered
\end{equation}
H\"older's integral inequality for functions $f\in L^{p_1}[0,l]$ and $g\in L^{p_2}[0,l]$
\begin{equation}\label{Holder}
\gathered
\Big|\int\limits_{0}^lf(t)g(t)\,dt\Big|\le\Big(\int\limits_0^l|f(t)|^{p_1}\,dt\Big)^{1/p_1}\Big(\int\limits_0^l|g(t)|^{p_2}\,dt\Big)^{1/p_2},\\
\quad p_1\in(1,\infty),\quad \frac{1}{p_1}+\frac{1}{p_2}=1,
\endgathered
\end{equation}
in the particular case $p_1=p_2=1/2$ it is also called the Cauchy-Bunyakovsky inequality.
We will also use the following elementary inequality
\begin{equation}\label{ElN}
\gathered
-az^2+bz\le -\frac{a}{2}z^2+\frac{b^2}{2a},\quad a>0.
\endgathered
\end{equation}

For the poof of the existence of solutions to our problem we will use the following
\begin{theorem}[Theorem 1.4 in \cite{babin-v}]\label{teorema-o-kompakt}
Let $\mathcal X_0$, $\mathcal X$, $\mathcal X_1$ be Banach spaces with
$\mathcal X_1\subset \mathcal X\subset \mathcal X_0$, where $\mathcal X_0$ and $\mathcal X_1$ are reflexive and the embedding $\mathcal X_1\subset \mathcal X$ is compact. Let $\{v_n\}_{n=1}^{\infty}\subset L^{p_0}([0,T],\mathcal X_1)$, $1\le p_0<\infty$ be a bounded sequence in $ L^{p_0}([0,T],\mathcal X_1)$. Let the sequence of the generalized derivatives be such that
$\partial_t v_n\in L^{p_1}([0,T],\mathcal X_0)$, $p_1>1$ and bounded in $L^{p_1}([0,T],\mathcal X_0)$.
Then there is a subsequence of  $\{v_n\}_{n=1}^{\infty}$ which converges in $L^{p_0}([0,T],\mathcal X)$.
\end{theorem}

We will also use the integration by parts formula, see (33) on page 43 in \cite{babin-v}, as follows. Let $u,v\in L^p([0,T],L^p[0,l])$, $p\ge 2$, $\partial_t u,\partial_t v\in L^{p^{\prime}}([0,T],L^{p^{\prime}}[0,l])$
$1/p+1/p^{\prime}=1$, $L^{p}[0,l]\subset L^2[0,l]\subset L^{p^{\prime}}[0,l]$,  $u,v\in L^{\infty}([0,T],L^2[0,l])$, $0\le\tau\le T$, then
\begin{equation}\label{0.2}
\gathered
\int\limits_0^{\tau}\langle\partial_t u,v\rangle\,dt+\int\limits_0^{\tau}\langle u,\partial_tv\rangle
=\langle u(\tau),v(\tau)\rangle-\langle u(0),v(0)\rangle.
\endgathered
\end{equation}
%%%%%%%%%%%%%%%%%%%%%%%%%%%%%%%%%%%%%%%%%%%%%%%%%%%%%%%%%
\section{Problem statement}\label{problem}
%%%%%%%%%%%%%%%%%%%%%%%%%%%%%%%%%%%%%%%%%%%%%%%%%%%%%%%%%
Consider a coupled system of the following differential equations
\begin{equation}\label{1}
\gathered
u_t(z,t)=a^2u_{zz}(z,t)+f(u(z,t))+B^{\T}(z)x(t),\\
\dot x(t)=Cx(t)+X(x(t))+\int\limits_0^lD(z)u(z,t)\,dz,
\endgathered
\end{equation}
with initial conditions
\begin{equation}\label{2}
x(0)=x_0\in\mathbb{R}^n,\quad u(z,0)=\vph(z),\quad \vph\in L^2[0,l],\quad z\in(0,l),\quad t\in(0,+\infty)
\end{equation}
and boundary conditions
\begin{equation}\label{3}
u(0,t)=d_1(t),\quad u(l,t)=d_2(t),
\end{equation}
where $a>0$, $C\in\mathbb{R}^{n\times n}$,
$B,D\in L^2([0,l],\mathbb{R}^{n})$. For the nonlinear functions $f\,\,:\,\,\mathbb{R}\to\mathbb{R}$ and $X\,\,:\,\,\mathbb{R}^n\to\mathbb{R}^n$ we assume:

1) $f(s)$ can be written as $f(s)=f_0(s)+f_1(s)$, $f_0(0)=f_1(0)=0$, where $f_0(\cdot)$
is globally Lipschitz with Lipschitz constant $L>0$, that is
\begin{equation}\label{3*}
|f_0(s_2)-f_0(s_1)|\le L|s_2-s_1|,\quad s_1, s_2\in\mathbb{R}
\end{equation}
and $f_1\in C^{1}(\mathbb{R},\mathbb{R})$.

2)  $\exists\,\sigma\in\mathbb{R}$, $\alpha>0$ and $q\ge 3/2$ such that
\begin{equation}\label{3**}
sf(s)\le\sigma s^2-\alpha |s|^{2q}\quad\text{for all}\quad s\ne 0;
\end{equation}

3) $\exists\,\zeta>0$, $c_0>0$ such that for all  $s\in\mathbb{R}$ it holds that:
\begin{equation}\label{3***}
|f_1^{\prime}(s)|\le c_0(1+|s|^{2q-2}),\quad |f_1(s)|\le\zeta|s|^{2q-1}
\end{equation}

4) $X\in C^1(\mathbb{R}^n,\mathbb{R}^n)$ and there exists $\delta_2>0$ such that 
\begin{equation}\label{3****}
\|X(x)\|\le\delta_2\|x\|^{2q-1},\quad x\in\mathbb{R}^n.
\end{equation}

The disturbances $d_i(t)$, $i=1,2$ are assumed to be of class $L^{\infty}(\mathbb{R}_+)$ with  $d_i(0)=0$ and such that
$\frac{d}{dt}d_i\in L^{\infty}(\mathbb{R}_+)$.
\begin{remark}
A very special case of the problem \eqref{1}-\eqref{3} was considered in \cite{mart-sl}, namely with $f=0$, $X=0$ and $d_1=d_2=0$, which leads to a linear system without disturbances. The construction of a Lyapunov function from \cite{mart-sl} is not suitable in the case of \eqref{1}-\eqref{3}.
\end{remark}
For the problem \eqref{1}--\eqref{3} we are interested in weak solutions defined as follows
\begin{definition}\label{definition1}
A pair of functions $u\in L^2([0,T],H^{1}[0,l])\cap L^{2q}([0,T],L^{2q}[0,l])\cap C([0,T],L^{2}[0,l])$ and $x\in C([0,T],\mathbb{R}^n)$ satisfying
\begin{equation}\label{1bis}
\gathered
-\int\limits_{0}^T\langle u(\cdot,t),\phi_t(\cdot,t)\rangle\,dt=\int\limits_0^T(-a^2\langle u_{z}(\cdot,t),\phi_z(\cdot,t)\rangle\\
+\langle f(u(\cdot,t))+B^{\T}(\cdot)x(t),\phi(\cdot,t)\rangle)\,dt\\
\dot x(t)=Cx(t)+X(x)+\int\limits_0^lD(z)u(z,t)\,dz,
\endgathered
\end{equation}
for any $\phi\in C_0^{\infty}([0,T],C_0^{\infty}[0,l])$ and conditions \eqref{2}--\eqref{3}   is called weak solution of \eqref{1}--\eqref{3}.
\end{definition}
\begin{remark}
In this definition we mean that $u$ satisfies \eqref{2} in the sense that $\lim_{t\to 0+}\|u(z,t)-\varphi(z)\|_{L^2[0,l]}=0$. As well we mean that \eqref{3} holds for almost all $t\in(0,\infty)$, where the values $u(0,t)$ and $u(l,t)$ are well defined due to the compact embedding $H^1(0,l)\subset C[0,l]$.
\end{remark}
\begin{definition}\label{definition2}
	 We say that system \eqref{1}--\eqref{3} is input-to-state stable (ISS) from $(d_1,d_2)$ to $(u,x)$, if there exist $\beta_i\in \mathcal K\mathcal K\mathcal L$, $\gamma_i\in\mathcal K$, $i=1,2$ such that for any initial state $(\varphi,x_0)$ and any disturbance $(d_1,d_2)$ the solution satisfies
\begin{equation}\label{ISS}
\gathered
\|u(\cdot,t)\|_{L^2[0,l]}\le\beta_1(\|\vph\|_{L^2[0,l]},\|x_0\|,t)+\gamma_1(d_{\infty}),\quad t\ge0\\
\|x(t)\|\le\beta_2(\|\vph\|_{L^2[0,l]},\|x_0\|,t)+\gamma_2(d_{\infty}),\quad t\ge0
\endgathered
\end{equation}
where $d_{\infty}:=\max(\|d_1\|_{L^{\infty}(\Bbb R_+)},\|d_2\|_{L^{\infty}(\Bbb R_+)})$.
\end{definition}
The aim of this paper is to establish the ISS property of the interconnected system \eqref{1}--\eqref{3} in the sense of this definition.
%%%%%%%%%%%%%%%%%%%%%%%%%%%%%%%%%%%%%%%%
\section{Preliminary results}
%%%%%%%%%%%%%%%%%%%%%%%%%%%%%%%%%%%%%%%%
To construct a suitable Lyapunov function we first prove the following:
\begin{lemma}
Let $(u,x)\in (L^{2}([0,T],H^1[0,l])\cap L^{2q}([0,T],L^{2q}[0,l]))\cap C([0,T],L^2[0,l])\times C([0,T],\mathbb{R}^n)$ be a solution of the system \eqref{1}, then
\begin{equation}\label{5}
u(z,t)=w(z,t)+v(z,t),
\end{equation}
where $(v,x)\in (L^{2}([0,T],H_0^1[0,l])\cap L^{2q}([0,T],L^{2q}[0,l]))\cap C([0,T],L^2[0,l])\times C([0,T],\mathbb{R}^n)$ is the solution of
\begin{equation}\label{6}
\gathered
v_t(z,t)=a^2v_{zz}(z,t)+f(v(z,t))+B^{\T}(z)x(t)+g(z,t),\\
\dot x(t)=Cx(t)+X(x(t))+\int\limits_0^lD(z)v(z,t)\,dz+p(t).
\endgathered
\end{equation}
with initial condition
\begin{equation}\label{6*}
x(0)=x_0,\quad v(z,0)=\varphi(z)
\end{equation}
and boundary condition
\begin{equation}\label{6**}
v(0,t)=0,\quad v(l,t)=0,
\end{equation}
for  some functions
$w(z,t)$, $g(z,t)$ and $p(t)$ (provided explicitely in the proof) satisfying for almost all $(z,t)\in[0,l]\times\mathbb{R}_+$ the following estimates:
$|w(z,t)|\le d_{\infty}$  and
\begin{equation}\label{7}
|g(z,t)|\le \psi_1(\varepsilon)|v(z,t)|^{2q-1}+\psi_0(\varepsilon,d_{\infty}),\quad \|p(t)\|\le\sqrt{l}\|D\|_{L^2[0,l]}d_{\infty},
\end{equation}
where $\varepsilon>0$ and
\begin{equation*}
\psi_0(\varepsilon,d_{\infty})=2^{2q-3}c_0\varepsilon^{1-2q}\frac{1}{2q-1}d_{\infty}^{2q-1}+(L+c_0)d_{\infty}+2^{2q-3}c_0d_{\infty}^{2q-1},
\end{equation*}
\begin{equation*}
\psi_1(\varepsilon)=2^{2q-3}c_0\frac{2q-2}{2q-1}\varepsilon^{(2q-1)/(2q-2)}.
\end{equation*}
	\end{lemma}
\begin{proof}

Substituting $u(z,t)=w(z,t)+v(z,t)$ into \eqref{1} we obtain
\begin{equation}\label{8}
\gathered
v_t(z,t)=a^2v_{zz}(z,t)+B^{\T}(z)x(t)+a^2w_{zz}(z,t)-w_t(z,t)+f(v(z,t)+w(z,t)),\\
\dot x(t)=Cx(t)+X(x(t))+\int\limits_0^lD(z)v(z,t)\,dz+\int\limits_0^lD(z)w(z,t)\,dz.
\endgathered
\end{equation}
Let us denote $g(z,t):=f(v(z,t)+w(z,t))-f(v(z,t))$, $p(t):=\int\limits_0^lD(z)w(z,t)\,dz$.

Let $w\in H^{1,1}([0,l]\times[0,T])$ be the weak solution to
\begin{equation}\label{weak1}
\gathered
w_t(z,t)-a^2 w_{zz}(z,t)=0,\\
w(0,t)=d_1(t),\quad w(l,t)=d_2(t),\quad w(z,0)=0,
\endgathered
\end{equation}
by definition this means  that for any $\phi\in C_0^\infty ([0,l]\times[0,T])$ the following equality holds
\begin{equation}\label{weak1bis}
\gathered
\int\limits_0^T\langle w_t(\cdot,t),\phi(\cdot,t)\rangle\,dt+a^2\int\limits_0^T\langle w_{z}(z,t),\phi_z(z,t)\rangle\,dt=0.
\endgathered
\end{equation}
In this case \eqref{weak1bis} holds also for all $\phi\in H^{1,1}_0([0,l]\times[0,T])$.

Let us note that conditions $d_i(0)=0$, $\dot d_i\in L^{\infty}(\mathbb{R}_+)$ guarantee the existence of a weak solution  $w\in L^{2}([0,T],H^1[0,l])\cap L^{2q}([0,T],L^{2q}[0,l])$ to the problem  \eqref{weak1}.
Indeed, the function $\wt w(z,t)=w(z,t)-\Big(\frac{z}{l}d_2(t)+\frac{l-z}{l}d_1(t)\Big)$ is the weak solution to the problem
\begin{equation}\label{weak2}
\gathered
\wt w_t(z,t)-a^2\wt w_{zz}(z,t)=\Delta(z,t),\quad \Delta(z,t)=-\frac{z}{l}\dot d_2(t)-\frac{l-z}{l}\dot d_1(t),\\
\wt w(0,t)=0,\quad \wt w(l,t)=0,\quad \wt w(z,0)=0.
\endgathered
\end{equation}
The problems \eqref{weak1} and \eqref{weak2} are equivalent. Hence it is enough to show the existence of a weak solution to the problem \eqref{weak2}.
Since $\Delta\in L^2([0,l]\times[0,T])$, then by Theorem 4 from \cite{Mik78} (Chapter 6) the existence of a unique solution $\tilde w\in H^{2,1}([0,l]\times[0,T])$ to the problem \eqref{weak2} follows.

By the (weak) maximum principle, see e.g., \S30 in \cite{Vla71}  applied to $w$ it follows that
\begin{equation*}
\gathered
|w(z,t)|\le d_{\infty}\quad\text{for almost all}\quad (z,t)\in[0,l]\times[0,\infty).
\endgathered
\end{equation*}
From the assumption 1)---2) for $f$ we obtain the estimation of $g(z,t)$ for almost all $(z,t)\in[0,l]\times[0,\infty)$:
\begin{equation*}
\gathered
|g(z,t)|\le |f_0(v(z,t)+w(z,t))-f_0(v(z,t))|+|f_1(v(z,t)+w(z,t))-f_1(v(z,t))|\\
\le
L|w(z,t)|+|f^{\prime}(v(z,t)+\eta w(z,t))||w(z,t)|\\
\le Ld_{\infty}+c_0(1+(|v(z,t)|+\eta|w(z,t)|)^{2q-2})d_{\infty},
\endgathered
\end{equation*}
where $\eta\in(0,1)$. Since the function $s\mapsto 1+s^{2q-2},\; s\ge 0$, is convex, we can estimate further
\begin{equation*}
\gathered
|g(z,t)|
\le Ld_{\infty}+c_0(1+(|v(z,t)|+\eta|w(z,t)|)^{2q-2})d_{\infty}\\
\le Ld_{\infty}+c_0(1+2^{2q-3}(|v(z,t)|^{2q-2}+|w(z,t)|^{2q-2}))d_{\infty}\\
\le 2^{2q-3}c_0d_{\infty}|v(z,t)|^{2q-2}+(L+c_0)d_{\infty}+2^{2q-3}c_0d_{\infty}^{2q-1}.
\endgathered
\end{equation*}
By the Young's inequality with  $p_1=(2q-1)/(2q-2)$, $p_2=2q-1$ we get
\begin{equation*}
d_{\infty}|v(z,t)|^{2q-2}\le \frac{2q-2}{2q-1}\varepsilon^{(2q-1)/(2q-2)}|v(z,t)|^{2q-1}+\varepsilon^{1-2q}\frac{1}{2q-1}d_{\infty}^{2q-1},
\end{equation*}
where $\varepsilon>0$.
Finally, we obtain that for almost all $(z,t)\in[0,l]\times[0,\infty)$ it holds that
\begin{equation*}
\gathered
|g(z,t)|
\le 2^{2q-3}c_0\frac{2q-2}{2q-1}\varepsilon^{(2q-1)/(2q-2)}|v(z,t)|^{2q-1}
+c_0\varepsilon^{1-2q}\frac{2^{2q-3}}{2q-1}d_{\infty}^{2q-1}\\
+(L+c_0)d_{\infty}+2^{2q-3}c_0d_{\infty}^{2q-1}
=\psi_1(\varepsilon)|v(z,t)|^{2q-1}+\psi_0(\varepsilon,d_{\infty})
\endgathered
\end{equation*}
To estimate $p(t)$ we use the Cauchy-Bunyakovsky inequality:
\begin{equation*}
\|p(t)\|\le\int\limits_0^l\|D(z)\||w(z,t)|\,dz\le\|D\|_{L^2[0,l]}\|w(\cdot,t)\|_{L^2[0,l]}\le d_{\infty}\sqrt{l}\|D\|_{L^2[0,l]}.
\end{equation*}

Taking into account that $w\in H^{2,1}([0,l]\times[0,T])$
from \eqref{8} follows that $v\in L^2([0,T],H_0^1[0,l])\cap L^{2q}([0,T],L^{2q}[0,l])$ is a weak solution to \eqref{6} satisfying \eqref{6*} and  \eqref{6**}.
This finishes the proof of the lemma.
\end{proof}
To derive estimates of solutions to \eqref{6} we use the direct method of Lyapunov. For this we define
\begin{equation}\label{11}
V(v(\cdot,t),x):=\langle v(\cdot,t),v(\cdot,t)\rangle+2\langle x^{\T}P_{12},v(\cdot,t)\rangle+x^{\T}Px,
\end{equation}
where $P_{12}\in (H^2_0([0,l])^n$ the solution to the boundary value problem
\begin{equation}\label{11bis}
a^2P_{12}^{\prime\prime}(z)+C^{\T}P_{12}(z)=-B(z)-PD(z),\quad P_{12}(0)=P_{12}(l)=0,
\end{equation}
and $P\in\mathbb{R}^{n\times n}$ is a symmetric positive definite matrix.
Also we denote
\begin{equation*}
\Pi_1=\begin{pmatrix}
1&-\|P_{12}\|_{L^2[0,l]}\\
-\|P_{12}\|_{L^2[0,l]}&\lm_{\min}(P)
\end{pmatrix},\quad
\Pi_2=\begin{pmatrix}
1&\|P_{12}\|_{L^2[0,l]}\\
\|P_{12}\|_{L^2[0,l]}&\lm_{\max}(P)
\end{pmatrix}
\end{equation*}
then $V(v(\cdot,t),x)$ can be estimated as
\begin{equation}\label{12}
\lm_{\min}(\Pi_1)(\|v(\cdot,t)\|_{L^2[0,l]}^2+\|x\|^2)\le V(v(\cdot,t),x)\le \lm_{\max}(\Pi_2)(\|v(\cdot,t)\|_{L^2[0,l]}^2+\|x\|^2)
\end{equation}
%%%%%%%%%%%%%%%%%%%%%%%%%%%%%%%%%%%%%%%%%%%%%%%
%\section{Solution to the boundary value problem \eqref{11bis}}
%%%%%%%%%%%%%%%%%%%%%%%%%%%%%%%%%%%%%%%%%%%%%%%
Now let us consider the derivation of $P_{12}(z)$ as a solution to \eqref{11bis}.
The equation  \eqref{11bis} for $P_{12}(z)$ can be written as
\begin{equation}\label{17}
P_{12}^{\prime\prime}(z)+\frac{1}{a^2}C^{\T} P_{12}(z)=F(z),\quad P_{12}(0)=P_{12}(l)=0,
\end{equation}
where $F(z):=-\frac{1}{a^2}(B(z)+PD(z))$.

Let us first consider the case $C=C^{\T}\succ 0$. Using the functional calculus of matrices we can state the following  

\begin{lemma} Let $C=C^{\T}\succ 0$, $\det\sin(\frac{1}{a}C^{1/2}l)\ne 0$, then the solution of
\begin{equation}\label{18}
\gathered
P_{12}^{\prime\prime}(z)+\frac{1}{a^2}C^{\T}P_{12}(z)=F(z),\quad P_{12}(0)=P_{12}(l)=0
\endgathered
\end{equation}
can be written as
\begin{equation}\label{19}
\gathered
P_{12}(z)=\int\limits_0^lG(z,\xi)F(\xi)d\xi,
\endgathered
\end{equation}
with the Green's function $G(z,\xi)$ given by
\begin{equation*}
G(z,\xi)=\begin{cases}
a\sin(\frac{1}{a}(C^{\T})^{1/2}\xi)\sin(\frac{1}{a}(C^{\T})^{1/2}(z-l))((C^{\T})^{1/2}\sin(\frac{1}{a}(C^{\T})^{1/2}l))^{-1}\\ \text{for } 0\le\xi\le z\le l;\\
a\sin(\frac{1}{a}(C^{\T})^{1/2}z)\sin(\frac{1}{a}(C^{\T})^{1/2}(\xi-l))((C^{\T})^{1/2}\sin(\frac{1}{a}(C^{\T})^{1/2}l))^{-1}\\ \text{for } 0\le z\le \xi\le l.
\end{cases}
\end{equation*}
\end{lemma}
\begin{proof} By the variation of constants method we look for a solution to
\eqref{18} in the form
\begin{equation*}
\gathered
P_{12}(z)=
\sin(\frac{1}{a}(C^{\T})^{1/2}z)C_1(z)+\cos(\frac{1}{a}(C^{\T})^{1/2}z)C_2(z)
%P_{12}^{\prime}(z)=\frac{1}{a}S_{\gamma}^{1/2}\Big(\cos(\frac{1}{a}S_{\gamma}^{1/2}z)C_1(z)-\sin(\frac{1}{a}S_{\gamma}^{1/2}z)C_2(z)\Big)
\endgathered
\end{equation*}
where vectors $C_1$ and $C_2$ satisfy
\begin{equation*}
\sin(\frac{1}{a}C^{\T})^{1/2}z)C_1^{\prime}(z)+\cos(\frac{1}{a}(C^{\T})^{1/2}z)C_2^{\prime}(z)=0.
\end{equation*}
Substituting this $P_{12}$ into \eqref{18} we obtain that $C_1$ and $C_2$ must satisfy
\begin{equation*}
\frac{1}{a}(C^{\T})^{1/2}\Big(\cos(\frac{1}{a}(C^{\T})^{1/2}z)C_1^{\prime}(z)-\sin(\frac{1}{a}(C^{\T})^{1/2}z)C_2^{\prime}(z)\Big)=F(z).
\end{equation*}
This implies that
\begin{equation*}
\gathered
C_1^{\prime}(z)=a(C^{\T})^{-1/2}\cos(\frac{1}{a}(C^{\T})^{1/2}z)F(z),\\
C_2^{\prime}(z)=-a(C^{\T})^{-1/2}\sin(\frac{1}{a}(C^{\T})^{1/2}z)F(z).
\endgathered
\end{equation*}
After integration we get
\begin{equation*}
\gathered
C_1(z)=C_1(0)+a(C^{\T})^{-1/2}\int\limits_0^z\cos(\frac{1}{a}(C^{\T})^{1/2}\xi)F(\xi)d\xi,\\
C_2(z)=C_2(0)-a(C^{\T})^{-1/2}\int\limits_0^z\sin(\frac{1}{a}(C^{\T})^{1/2}\xi)F(\zeta)d\xi
\endgathered
\end{equation*}
Consequently,
\begin{equation}\label{20}
\gathered
P_{12}(z)=
\sin(\frac{1}{a}(C^{\T})^{1/2}z)C_1(0)+\cos(\frac{1}{a}(C^{\T})^{1/2}z)C_2(0)\\
+a(C^{\T})^{-1/2}\int\limits_0^z\sin(\frac{1}{a}(C^{\T})^{1/2}(z-\xi))F(\xi)\,d\xi.
\endgathered
\end{equation}
From the boundary conditions $P_{12}(0)=0$, $P_{12}(l)=0$ follows
$C_2(0)=0$ and
\begin{equation*}
\gathered
C_1(0)=-a(C^{\T})^{-1/2}\int\limits_0^l(\sin(\frac{1}{a}(C^{\T})^{1/2}l))^{-1}\sin(\frac{1}{a}(C^{\T})^{1/2}(l-\xi))F(\xi)\,d\xi\\
=-a(C^{\T})^{-1/2}\int\limits_0^l\Big(\cos(\frac{1}{a}(C^{\T})^{1/2}\xi)-
\cot(\frac{1}{a}(C^{\T})^{1/2}l)\sin(\frac{1}{a}(C^{\T})^{1/2}\xi)\Big)F(\xi)\,d\xi
\endgathered
\end{equation*}
Substituting these $C_1(0)$ and $C_2(0)$ into  \eqref{20} we obtain \eqref{19}.
By properties of the Green's function it follows that this $P_{12}$ defined by \eqref{19}
satisfies \eqref{18}.
\end{proof}
Now we consider the case when $C^{\T}=C\succ 0$ is not true. In this case the \eqref{19} is the solution to \eqref{18}, if in the expression for the Green-function we adopt that
\begin{equation*}
\gathered
(C^{\T})^{1/2}\sin\Big(\frac{l}{a}(C^{\T})^{1/2}\Big)=\sum\limits_{k=0}^{\infty}\frac{(-1)^{k}(l/a)^{2k+1}}{(2k+1)!}(C^{\T})^{k+1},\\
\sin((C^{\T})^{1/2}\frac{\xi}{a})\sin((C^{\T})^{1/2}\frac{(z-l)}{a})\\
=\sum\limits_{k=0}^{\infty}\sum\limits_{p=0}^{\infty}\frac{(-1)^{k+p}}{(2k+1)!(2p+1)!}\Big(\frac{\xi}{a}\Big)^{2k+1}\Big(\frac{z-l}{a}\Big)^{2p+1}(C^{\T})^{k+p+1}.
\endgathered
\end{equation*}
These series can be calculated as follows. Let $T$ be nonsingular matrix such that $T^{-1}C^{\T}T=\text{diag}\{J_{n_1}(\lambda_1),\dots,J_{n_r}(\lambda_r)\}$ is the Jordan normal form of $C^{\T}$.
Then these series can be calculated separately for each Jordan block  $J_{n_k}(\lambda_k)$ explicitly.

%%%%%%%%%%%%%%%%%%%%%%%%%%%%%%%%%%%%%%%%%
\section{Main results}\label{sec:main}
%%%%%%%%%%%%%%%%%%%%%%%%%%%%%%%%%%%%%%%%%
Our first result provides conditions to guarantee that our system is well-posed. The second result establishes the ISS property.
\begin{theorem}\label{main-theorem}
Assume that

1) $f$ and $X$ satisfy conditions 1)---4) from section 2 and  $f^{\prime}_1(s)<0$ $\forall$ $s\in\mathbb{R}$, $s\ne 0$;

2) $\Pi_1$ is positive definite and for some $\delta_1>0$
\begin{equation*}
\gathered
x^{\T}PX(x)\le-\delta_1\|x\|^{2q},\quad x\in\mathbb{R}^n,
\endgathered
\end{equation*}

3) the following conditions are satisfied
\begin{equation}\label{24}
\gathered
\omega:=2\Big(\frac{\pi^2a^2}{l^2}-\|D\|_{L^2[0,l]}\|P_{12}\|_{L^2[0,l]}-\sigma\Big)>0,\\
\Omega:=-\Big(C^{\T}P+PC
-\frac{1}{a^2}\int\limits_0^l\int\limits_0^l(G(z,\zeta)(B(\zeta)+PD(\zeta))B^{\T}(z)\\
+B(z)(B^{\T}(\zeta)+D^{\T}(\zeta)P)G^{\T}(z,\zeta))d\zeta dz\Big)\succ 0,
\endgathered
\end{equation}

4) the matrix
\begin{equation*}
\gathered
\Xi=\begin{pmatrix}
\omega&-L\|P_{12}\|_{L^2[0,l]}\\
-L\|P_{12}\|_{L^2[0,l]}&\lambda_{\min}(\Omega)
\end{pmatrix}
\endgathered
\end{equation*}
is positive definite.

5) Let (the unique pair) $\tau_1>0$ and $\tau_2>0$ satisfying
$$\frac{\zeta\|P_{12}\|_{L^1[0,l]}}{q}\tau_1^{2q}+\frac{\delta_2\|P_{12}\|_{L^2[0,l]}(2q-1)}{q}\tau_1^{2q/(2q-1)}=2\delta_1$$
$$\frac{\delta_2\|P_{12}\|_{L^2[0,l]}l^{q-1}}{q}\tau_2^{-2q}+\frac{\zeta(2q-1)\|P_{12}\|_{L^{\infty}[0,l]}}{q}\tau_2^{-2q/(2q-1)}=2\alpha$$
%\begin{equation}\label{24*}
%\gathered
%frac{\zeta\|P_{12}\|_{L^1[0,l]}}{q}\tau_1^{2q}+\frac{\delta_2\|P_{12}\|_{L^2[0,l]}(2q-1)}{q}\tau_1^{2q/(2q-1)}=2\delta_1,\\
%frac{\delta_2\|P_{12}\|_{L^2[0,l]}l^{q-1}}{q}\tau_2^{-2q}+\frac{\zeta(2q-1)\|P_{12}\|_{L^{\infty}[0,l]}}{q}\tau_2^{-2q/(2q-1)}=2\alpha
%endgathered
%end{equation}
be such that $\tau_2<\tau_1$.

Then there exist $u\in L^{2}([0,T],H^1[0,l])\cap L^{2q}([0,T],L^{2q}[0,l])$ and $x\in C([0,T],\mathbb{R}^n)$ solving the problem \eqref{1} -- \eqref{3} for any $T>0$.
\end{theorem}
\begin{remark}
Note that $u\in L^{2}([0,T],H^1[0,l])\cap L^{2q}([0,T],L^{2q}[0,l])$ and $\partial_t u\in
L^{2}([0,T],H^{-1}[0,l])\cap L^{\frac{2q}{2q-1}}([0,T],L^{\frac{2q}{2q-1}}[0,l])$ and hence by \cite{babin-v} $u$ is absolutely continuous. 
\end{remark}
The proof follows the ideas of the proof of Theorem 3.1 in \cite{babin-v} on page 38 (see also \cite{lions}) and is based on the Lyapunov function defined in Section 4. We postpone the detailed proof to the Appendix A and state our second main result as follows,

\begin{theorem}\label{th2}
	Let \eqref{1}--\eqref{3} be such that $f$ and $X(x)$ satisfy assumptions 1)--4) of Section 2 and let conditions 2)--5) from Theorem \ref{main-theorem} be satisfied.
Assume that for any initial condition \eqref{2} and disturbances \eqref{3}  there exist a solution $u\in L^{2}([0,T],H^1[0,l])\cap L^{2q}([0,T],L^{2q}[0,l])$ and $x\in C([0,T],\mathbb{R}^n)$  for any $T>0$. Then  \eqref{1}--\eqref{3} is ISS from $(d_1,d_2)$ to $(u,x)$.
\end{theorem}
\begin{remark}
	Let us note that conditions  2)--5) from Theorem \ref{main-theorem}  can be verified by means of  standard numerical tools available in Matlab or Maple.
\end{remark}
\begin{proof} 
	Using Lemma 2.2 from \cite{DKS20} we calculate the time derivative of  $V(v(\cdot,t),x)$ along solutions to
	\eqref{6}:
	\begin{equation}\label{13}
	\dot V(v(\cdot,t),x)=W_1+W_2+W_3+W_4,
	\end{equation}
	where we have denoted (see detailed calculations in Appendix \ref{detailed-calculations})
	\begin{equation}\label{14}
	\gathered
	W_1=-2a^2\langle v_{z}(\cdot,t),v_{z}(\cdot,t)\rangle+2\int\limits_{0}^lv(z,t)f(v(z,t))\,dz\\
	+2\int\limits_0^lv(z,t)D^{\T}(z)\,dz\int\limits_0^lP_{12}(z)v(z,t)\,dz,
	\endgathered
	\end{equation}
	\begin{equation}\label{15}
	\gathered
	W_2=x^{\T}\Big(C^{\T}P+PC+\int\limits_{0}^l(P_{12}(z)B^{\T}(z)+B(z)P_{12}^{\T}(z))\,dz\Big)x\\
	+2x^{\T}PX(x),
	\endgathered
	\end{equation}
	\begin{equation}\label{16}
	\gathered
	W_3=2\int\limits_0^lv(z,t)g(z,t)\,dz+2p^{\T}(t)\int\limits_0^lP_{12}(z)v(z,t)\,dz\\
	+2x^{\T}\int\limits_{0}^lP_{12}(z)g(z,t)\,dz+2x^{\T}Pp(t),
	\endgathered
	\end{equation}
	\begin{equation}\label{16*}
	\gathered
	W_4=2X^{\T}(x)\int\limits_0^lP_{12}(z)v(z,t)\,dz+2x^{\T}\int\limits_0^lP_{12}(z)f(v(z,t))\,dz.
	\endgathered
	\end{equation}
Using the Friedrich's inequality
	\begin{equation*}
	\gathered
	-\langle v_{z}(\cdot,t),v_{z}(\cdot,t)\rangle=-\int\limits_0^lv_z^2(z,t)\,dz\le-\frac{\pi^2}{l^2}\|v(\cdot,t)\|_{L^2[0,l]}^2
	\endgathered
	\end{equation*}
	and the Cauchy-Bunyakovskiy inequality
	\begin{equation*}
	\gathered
	\int\limits_0^lv(z,t)D^{\T}(z)\,dz\int\limits_0^lP_{12}(z)v(z,t)\,dz\le\|D\|_{L^2[0,l]}\|P_{12}\|_{L^2[0,l]}\|v(\cdot,t)\|_{L^2[0,l]}^2
	\endgathered
	\end{equation*}
we can estimate $W_1$ with help of integration by parts as follows
	\begin{equation*}
	W_1\le \Big(-\frac{2\pi^2a^2}{l^2}+2\|D\|_{L^2[0,l]}\|P_{12}\|_{L^2[0,l]}\Big)\|v(\cdot,t)\|_{L^2[0,l]}^2+2\int\limits_{0}^lv(z,t)f(v(z,t))\,dz.
	\end{equation*}
	By the assumption 2) of Section 3 we obtain that
	\begin{equation}\label{16bis}
	W_1\le \Big(-\frac{2\pi^2a^2}{l^2}+2\|D\|_{L^2[0,l]}\|P_{12}\|_{L^2[0,l]}+2\sigma\Big)\|v(\cdot,t)\|_{L^2[0,l]}^2-2\alpha\|v(\cdot,t)\|_{L^{2q}[0,l]}^{2q}.
	\end{equation}
		Let us estimate $W_3$ in \eqref{13}. Applying the first inequality from \eqref{7} we obtain
\begin{equation*}
\gathered
\int\limits_0^l v(z,t)g(z,t)\,dz\le\psi_0(\varepsilon,d_{\infty})\int\limits_0^l|v(z,t)|\,dz+\psi_1(\varepsilon)\int\limits_0^l|v(z,t)|^{2q}\,dz\\
\le\psi_0(\varepsilon,d_{\infty})\sqrt{l}\|v(\cdot,t)\|_{L^2[0,l]}+\psi_1(\varepsilon)\|v(\cdot,t)\|_{L^{2q}[0,1]}^{2q}.
\endgathered
\end{equation*}
By the triangle inequality and the Cauchy-Bunyakovskiy inequality we have
\begin{equation*}
\gathered
\Big\|\int\limits_0^lP_{12}(z)g(z,t)\,dz\Big\|\le\int\limits_0^l\|P_{12}(z)\|(\psi_0(\varepsilon,d_{\infty})+\psi_1(\varepsilon)|v(z,t)|^{2q-1})\,dz\\
\le \sqrt{l}\|P_{12}\|_{L^2[0,l]}\psi_0(\varepsilon,d_{\infty})+\psi_1(\varepsilon)\int\limits_0^l\|P_{12}(z)\||v(z,t)|^{2q-1}\,dz,
\endgathered
\end{equation*}
and by the Cauchy-Bunyakovskiy inequality we estimate
\begin{equation*}
\gathered
x^{\T}\int\limits_0^lP_{12}(z)g(z,t)\,dz\\
\le
\sqrt{l}\|P_{12}\|_{L^2[0,l]}\psi_0(\varepsilon,d_{\infty})\|x\|+\psi_1(\varepsilon)\int\limits_0^l\|P_{12}(z)\|\|x\||v(z,t)|^{2q-1}\,dz .
\endgathered
\end{equation*}
By the Young's inequality \eqref{Yang} with $p_1=2q$, $p_2=2q/(2q-1)$, we have
\begin{equation*}
\|x\||v(z,t)|^{2q-1}
\le \frac{1}{2q}\|x\|^{2q}+\Big(1-\frac{1}{2q}\Big)|v(z,t)|^{2q}.
\end{equation*}
Hence we obtain
\begin{equation*}
\gathered
x^{\T}\int\limits_0^lP_{12}(z)g(z,t)\,dz\\
\le
\sqrt{l}\|P_{12}\|_{L^2[0,l]}\psi_0(\varepsilon,d_{\infty})\|x\|+
\frac{1}{2q}\psi_1(\varepsilon)\|P_{12}\|_{L^1[0,l]}\|x\|^{2q}\\
+\Big(1-\frac{1}{2q}\Big)\psi_1(\varepsilon)\|P_{12}\|_{L^{\infty}[0,l]}\|v(\cdot,t)\|_{L^{2q}[0,l]}^{2q}.
\endgathered
\end{equation*}
By the Cauchy-Bunyakovskiy inequality we have
\begin{equation*}
\gathered
\Big|p^{\T}(t)\int\limits_0^lP_{12}(z)v(z,t)\,dz\Big|\le
d_{\infty}\sqrt{l}\|D\|_{L^2[0,l]}\|P_{12}\|_{L^2[0,l]}\|v(\cdot,t)\|_{L^2[0,l]},
\endgathered
\end{equation*}
\begin{equation*}
\gathered
x^{\T}Pp(t)\le\sqrt{l}\|P\|d_{\infty}\|D\|_{L^2[0,l]}\|x\|.
\endgathered
\end{equation*}
Hence,
\begin{equation}\label{W3}
\gathered
W_3\le 2\sqrt{l}(\psi_0(\varepsilon,d_{\infty})+d_{\infty}\|D\|_{L^2[0,l]}\|P_{12}\|_{L^2[0,l]})\|v(\cdot,t)\|_{L^2[0,l]}\\
+2\sqrt{l}(\|P_{12}\|_{L^2[0,l]}\psi_0(\varepsilon,d_{\infty})+\|P\|\|D\|_{L^2[0,l]}d_{\infty})
\|x\|\\
+2\psi_1(\varepsilon)\Big(1+\frac{2q-1}{2q}\|P_{12}\|_{L^{\infty}[0,l]}\Big)\|v(\cdot,t)\|_{L^{2q}[0,l]}^{2q}
+\frac{1}{q}\psi_1(\varepsilon)\|P_{12}\|_{L^1[0,l]}\|x\|^{2q}.
\endgathered
\end{equation}
To estimate $W_4$ we apply the Cauchy-Bunyakovskiy inequality to \eqref{16*} and using \eqref{3*}-\eqref{3****} obtain
\begin{equation*}
\gathered
W_4\le 2\|X(x)\|\int\limits_0^l\|P_{12}(z)\||v(z,t)|\,dz\\
+2\|x\|\int\limits_0^l\|P_{12}(z)\||f(v(z,t))|\,dz\le 2\delta_2\|P_{12}\|_{L^2[0,l]}\|x\|^{2q-1}\|v(\cdot,t)\|_{L^2[0,l]}\\
+2L\|P_{12}\|_{L^2[0,l]}\|x\|\|v(\cdot,t)\|_{L^2[0,l]}+2\zeta\|x\|\int\limits_0^l\|P_{12}(z)\||v(z,t)|^{2q-1}dz.
\endgathered
\end{equation*}
By the Young's inequality \eqref{Yang} with $p_1=\frac{2q}{2q-1}$, $p_2=2q$, $\tau>0$ we have
\begin{equation*}
\gathered
\|x\|^{2q-1}\|v(\cdot,t)\|_{L^2[0,l]}\le\frac{2q-1}{2q}\tau^{2q/(2q-1)}\|x\|^{2q}
+\frac{1}{2q}\tau^{-2q}\|v(\cdot,t)\|_{L^2[0,l]}^{2q}
\endgathered
\end{equation*}
By the H\"older's inequality \eqref{Holder} we have
\begin{equation}\label{W3bis}
\|v(\cdot,t)\|_{L^2[0,l]}^{2q}=\Big(\int\limits_0^l|v(z,t)|^2\,dz\Big)^q\le l^{q-1}\int\limits_0^l|v(z,t)|^{2q}\,dz=l^{q-1}\|v(\cdot,t)\|_{L^{2q}[0,l]}^{2q},
\end{equation}
which implies that
\begin{equation*}
\|x\|^{2q-1}\|v(\cdot,t)\|_{L^2[0,l]}\le\frac{2q-1}{2q}\tau^{2q/(2q-1)}\|x\|^{2q}+\frac{l^{q-1}}{2q}\tau^{-2q}\|v(\cdot,t)\|_{L^{2q}[0,l]}^{2q}.
\end{equation*}
By the Young's inequality with $p_2=2q/(2q-1)$, $p_1=2q$, $\tau>0$ we have
\begin{equation*}
\|x\||v(z,t)|^{2q-1}\le\frac{1}{2q}\tau^{2q}\|x\|^{2q}+\frac{2q-1}{2q}\tau^{-2q/(2q-1)}|v(z,t)|^{2q}.
\end{equation*}
Finally $W_4$ can be estimated as follows
\begin{equation}\label{W4}
\gathered
W_4\le 2L\|P_{12}\|_{L^2[0,l]}\|x\|\|v(\cdot,t)\|_{L^2[0,l]}\\
+\Big(\frac{\zeta\|P_{12}\|_{L^1[0,l]}}{q}\tau^{2q}+\delta_2\|P_{12}\|_{L^2[0,l]}\frac{2q-1}{q}\tau^{2q/(2q-1)}\Big)\|x\|^{2q}\\
+\Big(\delta_2\|P_{12}\|_{L^2[0,l]}\frac{l^{q-1}}{q}\tau^{-2q}+\frac{\zeta(2q-1)\|P_{12}\|_{L^{\infty}[0,l]}}{q}\tau^{-2q/(2q-1)}\Big)
\|v(\cdot,t)\|_{L^{2q}[0,l]}^{2q}.
\endgathered
\end{equation}
From \eqref{15} and the assumptions of the theorem we can estimate $W_2$ as
\begin{equation}\label{W2}
W_2\le x^{\T}(t)\Omega x(t)-2\delta_1\|x(t)\|^{2q}.
\end{equation}
Now from \eqref{16bis}, \eqref{24}, \eqref{W2}, \eqref{W3} and \eqref{W4} we obtain
\begin{equation}\label{DLF}
\gathered
\dot V(v(\cdot,t),x(t))
\le -\omega\|v(\cdot,t)\|_{L^2[0,l]}^2-\lm_{\min}(\Omega)\|x(t)\|^2\\
+2L\|P_{12}\|_{L^2[0,l]}\|x(t)\|\|v(\cdot,t)\|_{L^2[0,l]}+
H_1(\varepsilon)\|v(\cdot,t)\|_{L^2[0,l]}\\+H_2(\varepsilon)\|x(t)\|+
H_3(\varepsilon,\tau)\|v(\cdot,t)\|_{L^{2q}[0,l]}^{2q}
+H_4(\varepsilon,\tau)\|x(t)\|^{2q},
\endgathered
\end{equation}
where we have denoted
\begin{equation*}
\gathered
H_1(\varepsilon)=2\sqrt{l}(\psi_0(\varepsilon,d_{\infty})+d_{\infty}\|D\|_{L^2[0,l]}\|P_{12}\|_{L^2[0,l]}),\\
H_2(\varepsilon)=2\sqrt{l}(\|P_{12}\|_{L^2[0,l]}\psi_0(\varepsilon,d_{\infty})+\|P\|\|D\|_{L^2[0,l]}d_{\infty}),\\
H_3(\varepsilon,\tau)=2\psi_1(\varepsilon)\Big(1+\frac{2q-1}{2q}\|P_{12}\|_{L^{\infty}[0,l]}\Big)\\
+\delta_2\|P_{12}\|_{L^2[0,l]}\frac{l^{q-1}}{q}\tau^{-2q}+\frac{\zeta(2q-1)\|P_{12}\|_{L^{\infty}[0,l]}}{q}\tau^{-2q/(2q-1)}-2\alpha,\\
H_4(\varepsilon,\tau)=\frac{1}{q}\psi_1(\varepsilon)\|P_{12}\|_{L^1[0,l]}\\
+\frac{\zeta\|P_{12}\|_{L^1[0,l]}}{q}\tau^{2q}+\delta_2\|P_{12}\|_{L^2[0,l]}\frac{2q-1}{q}\tau^{2q/(2q-1)}-2\delta_1.
\endgathered
\end{equation*}

By the assumption 5) of the theorem there exists $\tau\in(\tau_2,\tau_1)$ such that
\begin{equation}\label{tauINEQ}
\gathered
\frac{\delta_2\|P_{12}\|_{L^2[0,l]} l^{q-1}}{q}\tau^{-2q}+\frac{\zeta(2q-1)\|P_{12}\|_{L^{\infty}[0,l]}}{q}\tau^{-2q/(2q-1)}-2\alpha<0,\\
\frac{\zeta\|P_{12}\|_{L^1[0,l]}}{q}\tau^{2q}+\frac{\delta_2\|P_{12}\|_{L^2[0,l]}(2q-1)}{q}\tau^{2q/(2q-1)}-2\delta_1<0.
\endgathered
\end{equation}
Since $\psi_1(\varepsilon)\to 0$ for $\varepsilon\to 0+$, there is some small enough $\varepsilon>0$ such that $H_3(\varepsilon,\tau)<0$, $H_4(\varepsilon,\tau)<0$. Hence for this choice of  $(\varepsilon,\tau)$ we obtain from assumption 4) of Theorem \ref{main-theorem} that
\begin{equation*}
\gathered
\dot V(v(\cdot,t),x(t))\le -\lambda_{\min}(\Xi)(\|v(\cdot,t)\|_{L^2[0,l]}^2+\|x(t)\|^2)\\+
H_1(\varepsilon,\tau)\|v(\cdot,t)\|_{L^2[0,l]}+H_2(\varepsilon,\tau)\|x(t)\|.
\endgathered
\end{equation*}
Now applying  \eqref{ElN} we get
\begin{equation*}
\dot V(v(\cdot,t),x(t))\le -\frac{\lambda_{\min}(\Xi)}{2}(\|v(\cdot,t)\|_{L^2[0,l]}^2+\|x(t)\|^2)+
\frac{H_1^2(\varepsilon,\tau)}{2\lambda_{\min}(\Xi)}+\frac{H_2^2(\varepsilon,\tau)}{2\lambda_{\min}(\Xi)}.
\end{equation*}
We denote
\begin{equation*}
\theta_0:=\frac{\lambda_{\min}(\Xi)}{2\lambda_{\max}(\Pi_2)},\qquad \vartheta(d_{\infty}):=\frac{H_1^2(\varepsilon,\tau)}{2\lambda_{\min}(\Xi)}+\frac{H_2^2(\varepsilon,\tau)}{2\lambda_{\min}(\Xi)},
\end{equation*}
and write
\begin{equation*}
\dot V(v(\cdot,t),x(t))\le -\theta_0V(v(\cdot,t),x(t))+\vartheta(d_{\infty}).
\end{equation*}
By means of the Gr\"onwall's inequality and using\eqref{0.2} we obtain the estimates
\begin{equation*}
\gathered
\|x(t)\|\le\sqrt{\frac{\lambda_{\max}(\Pi_2)}{\lambda_{\min}(\Pi_1)}}\varrho(x_0,\varphi)e^{-\theta_0t/2}+
\sqrt{\frac{\vartheta(d_{\infty})}{\lambda_{\min}(\Pi_1)\theta_0}},\\
\|v(\cdot,t)\|_{L^2[0,l]}\le\sqrt{\frac{\lambda_{\max}(\Pi_2)}{\lambda_{\min}(\Pi_1)}}\varrho(x_0,\varphi)e^{-\theta_0t/2}+
\sqrt{\frac{\vartheta(d_{\infty})}{\lambda_{\min}(\Pi_1)\theta_0}},
\endgathered
\end{equation*}
where $\varrho(x_0,\varphi):=(\|x_0\|^2+\|\varphi\|_{L^2[0,l]}^2)^{1/2}$.\\
Recall that $u(z,t)=v(z,t)+w(z,t)$, hence finally the following ISS estimate holds
\begin{equation*}
\|u(\cdot,t)\|_{L^2[0,l]}\le\sqrt{\frac{\lambda_{\max}(\Pi_2)}{\lambda_{\min}(\Pi_1)}}\varrho(x_0,\varphi)e^{-\theta_0t/2}+
\sqrt{l}d_{\infty}+\sqrt{\frac{\vartheta(d_{\infty})}{\lambda_{\min}(\Pi_1)\theta_0}}.
\end{equation*}
\end{proof}

\begin{corollary}\label{cor1}
Let $f$ be globally Lipschitz so that for some $L>0$ we have
$$|f(s_1)-f(s_2)|\le L|s_1-s_2|,\quad s_1, s_2\in\mathbb{R}$$
and such that for all $s\in\mathbb{R}$, $s\ne 0$ it hold that
$sf(s)\le\sigma s^2$, $f(0)=0$ for some $\sigma\in\mathbb{R}$. Further, we assume that the conditions 2) -- 4) of Theorem \ref{main-theorem} are satisfied.
Then system \eqref{1} -- \eqref{3} is ISS from $(d_1,d_2)$ to $(u,x)$ so that
\begin{equation}\label{C25}
\gathered
\|x(t)\|\le \sqrt{\frac{\lambda_{\max}(\Pi_2)}{\lambda_{\min}(\Pi_1)}}\varrho(x_0,\vph)e^{-\frac{\theta t}{2}}+
\sqrt{\frac{\beta}{\theta \lambda_{\min}(\Pi_1)}}d_{\infty},\quad t\ge 0,\\
\|u(\cdot,t)\|_{L^2[0,l]}\le \sqrt{\frac{\lambda_{\max}(\Pi_2)}{\lambda_{\min}(\Pi_1)}}\varrho(x_0,\varphi)e^{-\frac{\theta t}{2}}+
\Big(\sqrt l+\sqrt{\frac{\beta}{\theta \lambda_{\min}(\Pi_1)}}\Big)d_{\infty},\quad t\ge 0,
\endgathered
\end{equation}
where
$$\theta:=\frac{\lambda_{\min}(\Xi)}{2\lambda_{\max}(\Pi_2)},\quad
\beta:=\frac{2l(K_1^2+K_2^2)}{\lambda_{\min}(\Xi)},\quad \varrho(x_0,\varphi):=(\|x_0\|^2+\|\varphi\|_{L^2[0,l]}^2)^{1/2}$$
$$K_1:=L+\|D\|_{L^2[0,l]}\|P_{12}\|_{L^2[0,l]},\quad K_2:=L\|P_{12}\|_{L^2[0,l]}+\|D\|_{L^2[0,l]}\|P\|.$$
\end{corollary}
\begin{proof}
In this case we have $\psi_0(\varepsilon,d_{\infty})=Ld_{\infty}$, $\psi_1(\varepsilon)\equiv 0$, $H_3\equiv 0$, $H_4\equiv 0$,
\begin{equation*}
\gathered
H_1(\varepsilon)=2\sqrt{l}d_{\infty}(L+\|D\|_{L^2[0,l]}\|P_{12}\|_{L^2[0,l]})=2\sqrt{l}d_{\infty}K_1,\\
H_2(\varepsilon)=2\sqrt{l}(\|P_{12}\|_{L^2[0,l]}Ld_{\infty}
+\|P\|\|D\|_{L^2[0,l]}d_{\infty})=2\sqrt{l}d_{\infty}K_2,
\endgathered
\end{equation*}
so that from \eqref{DLF} we have the estimate
\begin{equation}\label{26C}
\gathered
\dot V(v(\cdot,t),x(t))\le -\lambda_{\min}(\Xi)(\|v(\cdot,t)\|_{L^2[0,l]}^2+\|x(t)\|^2)\\
+2\sqrt{l}d_{\infty}(K_1\|v(\cdot,t)\|_{L^2[0,l]}+K_2\|x(t)\|).
\endgathered
\end{equation}
and by \eqref{ElN} we can write
\begin{equation*}
\dot V(v(\cdot,t),x(t))\le -\frac{\lambda_{\min}(\Xi)}{2}(\|v(\cdot,t)\|_{L^2[0,l]}^2+\|x(t)\|^2)+\beta d_{\infty}^2,
\end{equation*}
by means of \eqref{12} for $V(v(\cdot,t),x(t))$ we obtain the differential inequality
\begin{equation*}
\dot V(v(\cdot,t),x(t))\le -\theta V(v(\cdot,t),x(t))+\beta d_{\infty}^2,
\end{equation*}
which implies after the integration that
\begin{equation*}
V(v(\cdot,t),x(t))\le e^{-\theta t}V(\vph,x_0)+\frac{2\beta\lambda_{\max}(\Pi_2)}{\lambda_{\min}(\Xi)} d_{\infty}^2.
\end{equation*}
Using \eqref{12}, the elementary inequality $\sqrt{a+b}\le\sqrt{a}+\sqrt{b}$, $a$, $b\ge 0$ and  $u(t,z)=v(t,z)+w(t,z)$ we arrive at \eqref{C25}.
\end{proof}
\begin{corollary}\label{unique} Under the conditions of Theorem \ref{main-theorem} a solution $(u,x)$ to the problem \eqref{1} -- \eqref{3} is unique in the class $L^2(\mathbb{R}_+,L^{\infty}[0,l])\times C(\mathbb{R}_+,\mathbb{R}^n)$.
\end{corollary}
See Appendix B for the proof.
%%%%%%%%%%%%%%%%%%%%%%%%%%%%%
\section{Example}
%%%%%%%%%%%%%%%%%%%%%%%%%%%%%
To illustrate our results we consider an interconnection of a parabolic equation and an ordinary differential equation as follows
\begin{equation}\label{27}
\gathered
u_t(z,t)=a^2u_{zz}(z,t)+\sigma \sin u(z,t)+bx(t),\\
\dot x(t)=cx(t)+d\int\limits_0^lu(z,t)\,dz
\endgathered
\end{equation}
with initial conditions
\begin{equation}\label{27*}
\gathered
x(0)=x_0\in\mathbb{R},\quad u(z,0)=\vph(z),\quad \vph\in L^2([0,l]),\quad z\in(0,l),\quad t\in(0,+\infty)
\endgathered
\end{equation}
and boundary conditions
\begin{equation}\label{27**}
\gathered
u(0,t)=d_1(t),\quad u(l,t)=d_2(t).
\endgathered
\end{equation}
where $a$, $c$, $l$, $\sigma$ are positive numbers and  $b$, $d\in\mathbb{R}$.

The solution to the (scalar) problem \eqref{11bis} with (scalar) $P>0$ can be calculated as follows.
The problem \eqref{11bis} reads as
\begin{equation*}
	\gathered
	a^2P_{12}^{\prime\prime}(z)+cP_{12}(z)=-b-Pd, \quad P_{12}(0)=P_{12}(l)=0.
	\endgathered
\end{equation*}
Since $c>0$, the general solution to the differential equation is
\begin{equation*}
	\gathered
	P_{12}(z)=C_1\sin(\lambda z)+C_2\cos(\lambda z)-\frac{b+Pd}{c},\quad \lambda=\frac{\sqrt{c}}{a}.
	\endgathered
\end{equation*}
From the boundary conditions $P_{12}(0)=P_{12}(l)=0$ we get 
\begin{equation*}
	\gathered
	C_2-\frac{b+Pd}{c}=0,\\
	C_1\sin(\lambda l)+C_2\cos(\lambda l)-\frac{b+Pd}{c}=0.
	\endgathered
\end{equation*}
From this we derive
\begin{equation*}
	\gathered
	C_1=\frac{b+Pd}{c}\Big(\frac{1}{\sin(\lambda l)}-\cot(\lambda l)\Big)=\tan\big(\frac{\lambda l}{2}\big)\frac{b+Pd}{c}.
	\endgathered
\end{equation*}
So that 
\begin{equation*}\label{28}
	P_{12}(z)=\frac{b+Pd}{c}\Big(\cos(\lambda z)-1+\tan(\frac{\lambda l}{2})\sin(\lambda z)\Big).
\end{equation*}
Let us calculate $\|P_{12}\|_{L^2([0,l])}=\Big|\frac{b+Pd}{c}\Big|\chi$, where
\begin{equation*}
	\chi=\Big(\int\limits_0^l(\cos(\lambda z)-1+\tan\big(\frac{\lambda l}{2}\big)\sin(\lambda z))^2\,dz\Big)^{1/2}.
\end{equation*}
Under the integral we have
\begin{equation*}
	\gathered
	\cos(\lambda z)-1+\tan(\frac{\lambda l}{2})\sin(\lambda z)=
	-2\sin^2\big(\frac{\lambda z}{2}\big)+2\tan\big(\frac{\lambda l}{2}\big)\sin\big(\frac{\lambda z}{2}\big)
	\cos\big(\frac{\lambda z}{2}\big)\\
	=\frac{2\sin\big(\frac{\lambda z}{2}\big)}{\cos\big(\frac{\lambda l}{2}\big)}\Big(\sin\big(\frac{\lambda l}{2}\big)\cos\big(\frac{\lambda z}{2}\big)-
	\sin\big(\frac{\lambda z}{2}\big)\cos\big(\frac{\lambda l}{2}\big)\Big)\\
	=\frac{2\sin\big(\frac{\lambda z}{2}\big)\sin\big(\frac{\lambda(l-z)}{2}\big)}{\cos\big(\frac{\lambda l}{2}\big)}
	=\frac{\cos(\lambda(z-\frac{l}{2}))-\cos\big(\frac{\lambda l}{2}\big)}{\cos\big(\frac{\lambda l}{2}\big)}\\
	=\frac{2\sin(\frac{\lambda l}{2})(\cos(\lambda(z-\frac{l}{2}))-\cos\frac{\lambda l}{2})}{\sin(\lambda l)}
	=\frac{\sin(\lambda z)+\sin(\lambda(l-z))-\sin(\lambda l)}{\sin(\lambda l)}
	\endgathered
\end{equation*}
Hence for the integral we can write
\begin{equation*}
	\chi=\frac{1}{|\sin(\lambda l)|}\Big(\int\limits_0^l(\sin(\lambda z)+\sin(\lambda(l-z))-\sin(\lambda l))^2\,dz\Big)^{1/2}.
\end{equation*}
Obviously, $\|B\|_{L^2([0,1])}=\sqrt{l}|b|$ and$\|D\|_{L^2([0,1])}=\sqrt{l}|d|$, so that we can calculate the constants
\begin{equation*}
	\gathered
	\omega=2\Big(\frac{\pi^2a^2}{l^2}-\|D\|_{L^2([0,1])}\|P_{12}\|_{L^2([0,1])}-\sigma\Big)\\
	=2\Big(\frac{\pi^2a^2}{l^2}-\Big|\frac{(b+Pd)\sqrt{l}d}{c}\Big|\chi-\sigma\Big)
	\endgathered
\end{equation*}
\begin{equation*}
	\gathered
	\Omega=-\Big(C^{\T}P+PC+\int\limits_0^{\ell}(P_{12}(z)B^{\T}(z)+B(z)P_{12}(z))\,dz\Big)\\
	=-2\Big(cP+\frac{b(b+Pd)}{c}\int\limits_0^{l}(\cos(\lambda z)-1+\tan(\frac{\lambda l}{2})\sin(\lambda z))\,dz\Big).
	\endgathered
\end{equation*}
For the last integral we calculate
\begin{equation*}
	\gathered
	\int\limits_0^{l}(\cos(\lambda z)-1+\tan(\frac{\lambda l}{2})\sin(\lambda z))\,dz=
	\frac{\sin(\lambda z)}{\lambda}-z-\tan(\frac{\lambda l}{2})\frac{\cos(\lambda z)}{\lambda}\Big|_{z=0}^{l}\\
	=\frac{1}{\lambda}\Big(\sin(\lambda l)+\tan(\frac{\lambda l}{2})(1-\cos(\lambda l))\Big)-l
	=\frac{2\tan(\frac{\lambda l}{2})}{\lambda}-l:=\varkappa,
	\endgathered
\end{equation*}
which implies
\begin{equation*}
	\Omega
	=-2\Big(cP+\frac{b(b+Pd)\varkappa}{c}\Big).
\end{equation*}
By means of the Taylor expansion 
\begin{equation*}
	\tan x=x+\frac{x^3}{3}+\cdots
\end{equation*}
we can write
\begin{equation*}
	\varkappa\approx\frac{2}{\lambda}\Big(\frac{\lambda l}{2}+\frac{(\lambda l)^3}{3}\Big)-l=\frac{\lambda^2 l^3}{12}
\end{equation*}
and similarly
\begin{equation*}
	\gathered
	\chi\approx\frac{1}{|\sin(\lambda l)|}\Big(\int\limits_0^l(\lambda l-\frac{(\lambda l)^3}{6}+
	\lambda(z-l)-\frac{(\lambda(z-l))^3}{6}-\lambda z+\frac{(\lambda z)^3}{6} )^2\Big)^{1/2}\\
	\approx
	\frac{1}{\lambda l}\Big(\frac{\lambda^6l^2}{4}\int\limits_0^lz^2(l-z)^2\,dz\Big)^{1/2}=
	\frac{l^2\sqrt{l}\lambda^2}{2\sqrt{30}}.
	\endgathered
\end{equation*}
By Corollary \ref{cor1} the system \eqref{27} is ISS if the conditions 2)-5) of Theorem \ref{main-theorem} are satisfies, that is if there exists $P$ such that
\begin{equation}\label{29}
P>0,\quad\omega(P)>0,\quad \Omega(P)>0,\quad \chi\Big|\frac{b+dP}{c}\Big|<\sqrt{P}.
\end{equation}
Let $P=p^*$ satisfy the inequalities \eqref{29}, then
\begin{equation*}
	\gathered
	K_1=\sigma+\sqrt{l}\frac{|d(b+p^*d)|}{c}\chi,\quad K_2=\sigma\frac{|d(b+p^*d)|}{c}\chi+\sqrt{l}|d|,\\
%	\alpha=\frac{\omega(p^*)\Omega(p^*)}{\omega(p^*)+\Omega(p^*)},\quad \beta=l(\omega(p^*)+\Omega(p^*))\Big(\frac{K_1^2}{\omega^2(p^*)}+\frac{K_2^2}{\Omega^2(p^*)}\Big).
 \endgathered
\end{equation*}
In case $\lambda l\ll 1$ these conditions can be simplified essentially, as in this case
\begin{equation*}
	\varkappa\approx\frac{\lambda^2l^3}{12},\quad \chi\approx\frac{\lambda^2l^2\sqrt{l}}{2\sqrt{30}}
\end{equation*}
and
\begin{equation*}
	\gathered
	\|P_{12}\|_{L^2[0,l]}\approx\Big|\frac{b+Pd}{c}\Big|\frac{\lambda^2l^2\sqrt{l}}{2\sqrt{30}},\quad \|B\|_{L^2[0,l]}=\sqrt{l}|b|,\quad \|D\|_{L^2[0,l]}=\sqrt{l}|d|,\\
	\omega=\omega(P)\approx 2\Big(\frac{\pi^2a^2}{l^2}-\frac{l^3|d(b+dP)|}{2\sqrt{30}a^2}-\sigma\Big),\\
	\Omega=\Omega(P)\approx-2(Pa^2\lambda^2+\frac{l^3b(b+Pd)}{12a^2}).
	\endgathered
\end{equation*}
Let the parameters of our system be given by
$c=0.25$, $a=1$, $l=1$, $b=1$, $d=-5$, $L=\sigma=1$ then we can take $P=p^*=1$, so that
\begin{equation*}
	\gathered
	\Xi=
	\begin{pmatrix}
		13.992949&-0.374626\\
		-0.374626&0.183765
	\end{pmatrix},\\
	\varkappa=0.021367,\quad \chi=0.023414,\quad \omega=13.992949,\quad\Omega=0.183766,\\
	\|P_{12}\|_{L^2[0,1]}=0.374626,\quad\lambda_{\min}(\Pi_1)=\lambda_{\min}(\Pi_2)=0.625374,\\
	\lambda_{\max}(\Pi_1)=\lambda_{\max}(\Pi_2)=1.374626,\quad K_1=2.873130,\quad K_2=8.7462728,\\
	\lambda_{\min}(\Xi)=0.1736102,\quad\theta=0.06315,\quad\beta=785.0749.
	\endgathered
\end{equation*}
By Corollary \ref{cor1} we obtain the following ISS estimates for the solutions of \eqref{27}
\begin{equation*}
	\gathered
	\|x(t)\|\le 1.482594 e^{-0.0315 t}+95.26d_{\infty},\quad t\ge 0,\\
	\|u(\cdot,t)\|_{L^2[0,1]}\le 1.482594 e^{-0.0315 t}+96.26d_{\infty}\quad t\ge 0.
	\endgathered
\end{equation*}
\begin{remark}
Let us not that the ODE subsystem is not ISS, because it is not globally asymptotically stable already in the disconnected case ($u=0$), hence there is no way to establish the ISS property for the interconnections by means of the small-gain theory or with help of vector Lyapunov functions.
\end{remark}

\section{Conclusion and future work}
\label{sec:conclusions}
In this work we have developed an approach for a construction of a Lyapunov function. With help of this function we have proved the ISS property of a nonlinear coupled systems of an ODE and a PDE with disturbances at the boundary. ISS-type estimation for solutions is also derived. Recall that in contrary to the most related works we allow the situation, where the decoupled ODE can be unstable. Also our problem is not self-adjoint, which is different from many of exiting works dealing with the ISS-like properties.

An interesting direction for future research is to extend the developed approach to the multidimensional case and to the case of coupled   PDE systems with time varying coefficients, where the decoupled PDEs are not necessarily stable. Furthermore, it is of interest to consider other types of boundary conditions, e.g., of Neumann or Robin type.

\appendix
\section{Proof of Theorem \ref{main-theorem}}
Here we will prove Theorem \ref{main-theorem} on the existence of solutions to the problem \eqref{1} -- \eqref{3}.
We will use the Lyapunov function $V$ from Section 3 and the ideas form \cite{babin-v,lions}. Hence we begin with the following change of variables
\begin{equation}\label{A2}
\wt u(z,t)=u(z,t)-H(z,t),\quad H(z,t)=\frac{z}{l}d_2(t)+\frac{l-z}{l}d_1(t)
\end{equation}
for which the problem  \eqref{1} -- \eqref{3} transforms to the equivalent problem
\begin{equation}\label{A3}
\gathered
\wt u_t(z,t)=a^2\wt u_{zz}(z,t)+f(\wt u(z,t)+H(z,t))+B^{\T}(z)x(t)-H_t(z,t)\,\\
\dot x(t)=Cx(t)+X(x(t))+\int\limits_0^lD(z)\wt u(z,t)\,dz+\int\limits_0^lD(z)H(z,t)\,dz
\endgathered
\end{equation}
with initial states $x(0)=x_0\in\mathbb{R}^n$, $\wt u(z,0)=\vph(z)$, $\vph\in L^2[0,l]$, $z\in(0,l)$, $t\in(0,\infty)$
and boundary conditions
$\wt u(0,t)=0$, $\wt u(l,t)=0$.

In the sequel we drop the symbol $\,\,\wt{}\,\,$ over $u$ to simplify notation. Consider the orthonormal basis $e_j(z)=\sqrt{\frac{2}{l}}\sin\frac{\pi jz}{l}$, $j\in\mathbb{N}$ in $L^2[0,l]$.

We define the following projection operator $\Pi_Nf:=\sum\limits_{p=0}^Nf_p e_p(z)$, $f_p=(f,e_p)_{L^2[0,l]}$  that maps the Hilbert space $L^2[0,l]$ onto the finite dimensional subspace
$E_N=\text{span}\{e_1(z),\dots,e_N(z)\}$.

The Galerkin system corresponding to the problem \eqref{A3} is as follows
\begin{equation}\label{A4}
\gathered
\partial_t u_N(z,t)=a^2\partial_{zz}u_{N}(z,t)+\Pi_Nf(u_N(z,t)+H(z,t))\\
+\Pi_NB^{\T}(z)x_N(t)-\Pi_NH_t(z,t)\,\\
\dot x_N(t)=Cx_N(t)+X(x_N(t))+\int\limits_0^lD(z)u_N(z,t)\,dz+\int\limits_0^lD(z)H(z,t)\,dz
\endgathered
\end{equation}
with initial conditions $u_N(z,0)=\Pi_N\vph(z)$, $x_N(0)=x_0$. The boundary conditions
$u_N(0,t)=u_N(l,t)=0$ are satisfied since $u_N(z,t)$ is an element of $E_N$.
Due to the assumptions 1)--4) introduced in Section \ref{problem} it follows that the ODE-system \eqref{A4} satisfies the local Lipschitz condition and hence possesses a solution defined for $t\in[0,t_{N}],\; t_{N}>0$.

To derive a priori estimates for the solutions $(u_N(z,t),x_N(t))$ of the ODE system \eqref{A4} we use the function $V$ from Section 4
\begin{equation}\label{A5}
V(u_N(\cdot,t),x_N)=\int\limits_0^l u_N^2(z,t)\,dz+2x^{\T}_N\int\limits_0^lP_{12}(z)u_N(z,t)\,dz+x_N^{\T}Px_N,
\end{equation}
for which along solutions of \eqref{A4} we calculate
\begin{equation}\label{A6}
\dot V(u_N(\cdot,t),x_N(t))=R_1(z,t)+R_2(z,t)+R_3(z,t)+R_4(z,t),
\end{equation}
where
\begin{equation}\label{A12}
\gathered
R_1(z,t)=-2a^2\int\limits_0^l(\partial_zu_N(z,t))^2\,dz+2\int\limits_0^lu_N(z,t)\Pi_Nf(u_N(z,t))\,dz\\
+2\int\limits_0^lu_N(z,t)D^{\T}(z)\,dz\int\limits_0^lP_{12}(z)u_N(z,t)\,dz
\endgathered
\end{equation}
\begin{equation}\label{A13}
\gathered
R_2(z,t)=2\int\limits_0^lu_N(z,t)(\Pi_N-I)B^{\T}(z)\,dzx_N(t)\\+2X^{\T}(x_N(t))\int\limits_0^lP_{12}(z)u_N(z,t)\,dz
+2x_N^{\T}(t)\int\limits_0^lP_{12}(z)\Pi_Nf(u_N(z,t))\,dz
\endgathered
\end{equation}
\begin{equation}\label{A14}
\gathered
R_3(z,t)=x_N^{\T}(t)(C^{\T}P+PC+\int\limits_0^l(P_{12}(z)B^{\T}(z)+B(z)P_{12}^{\T}(z))\,dz)x_N(t)\\
+2x_N^{\T}(t)PX(x_N(t))\\
+x_N^{\T}(t)\int\limits_0^l(P_{12}(z)(\Pi_N-I)B^{\T}(z)+(\Pi_N-I)B(z)P_{12}^{\T}(z))\,dzx_N(t).
\endgathered
\end{equation}
\begin{equation}\label{A15}
\gathered
R_4(z,t)=-2\int\limits_0^lu_N(z,t)\Pi_NH_t(z,t)\,dz+2\int\limits_0^lH(z,t)D^{\T}(z)\,dz\int\limits_0^lP_{12}(z)u_N(z,t)\,dz\\
-2x_N^{\T}(t)\int\limits_0^lP_{12}(z)\Pi_NH_t(z,t)\,dz+2x_N^{\T}(t)\int\limits_0^lPD(z)H(z,t)\,dz\\
+2\int\limits_0^lu_N(z,t)\Pi_N(f(u_N(z,t)+H(z,t))-f(u_N(z,t)))\,dz\\
+2x_N^{\T}(t)\int\limits_0^lP_{12}(z)\Pi_N(f(u_N(z,t)+H(z,t))-f(u_N(z,t)))\,dz.
\endgathered
\end{equation}
\begin{lemma}\label{lemmaA1}
There exists $N^*$ such that for any  $N\ge N^*$ and $t\in\mathbb{R}_+$ the solution $(u_N,x_N)$ to \eqref{A4} satisfies
\begin{equation}\label{A30bis}
\gathered
\|x_N(t)\|^2+\|u_N(\cdot,t)\|^2_{L^2[0,l]}\le e^{-\gamma_0t}\frac{\lambda_{\max}(\Pi_2)}{\lambda_{\min}(\Pi_1)}(\|x_0\|^2+\|\vph\|^2_{L^2[0,l]})
+\frac{\gamma_2}{\lambda_{\min}(\Pi_1)\gamma_0},
\endgathered
\end{equation}
where $\gamma_0$ and $\gamma_2$ are some positive constants. In particular the solution $(u_N,x_N)$ exists for all $t\ge 0$. Moreover for all $N\ge N^*$ and $T>0$ the next  a priori estimate is true
\begin{equation}\label{A30*}
\int\limits_0^T(\|u_N(\cdot,s)\|_{L^{2q}[0,l]}^{2q}+\|x_N(s)\|^{2q})\,ds\le\frac{\gamma_2T}{\gamma_1}+\frac{\lambda_{\max}(\Pi_2)}{\gamma_1}(\|x_0\|^2
+\|\vph\|_{L^2[0,l]}^2)
\end{equation}
for some constant $\gamma_1>0$.
\end{lemma}
\begin{proof}

Step by step we estimate parts of the expressions for $R_i(z,t)$, $i=1,2,3,4$.

By the Friedrich's inequality, taking the boundary values of $u_N$ we have
\begin{equation}\label{A15*}
-\int\limits_0^l(\partial_zu_N(z,t))^2\,dz\le -\frac{\pi^2}{l^2}\int\limits_0^lu_N^2(z,t)\,dz.
\end{equation}
Since the orthogonal projector $\Pi_N$ is self-adjoint and by the property \eqref{3**} we get
\begin{equation}\label{A16}
\gathered
\int\limits_0^lu_N(z,t)\Pi_Nf(u_N(z,t))\,dz=(u_N,\Pi_Nf(u_N))_{L^2[0,l]}=(\Pi_Nu_N,f(u_N))_{L^2[0,l]}\\
=(u_N,f(u_N))_{L^2[0,l]}\le\sigma\|u_N(\cdot,t)\|_{L^2[0,l]}^2-\alpha\|u_N(\cdot,t)\|_{L^{2q}[0,l]}^{2q}.
\endgathered
\end{equation}
By the Cauchy-Bunyakovskiy inequality we can estimate
\begin{equation}\label{A17}
\int\limits_0^lu_N(z,t)D^{\T}(z)\,dz\int\limits_0^lP_{12}(z)u_N(z,t)\,dz\le \|D\|_{L^2[0,l]}\|P_{12}\|_{L^2[0,l]}\|u_N(\cdot,t)\|_{L^2[0,l]}^2.
\end{equation}
Since $B\in (L^2[0,l])^n$, for any $\eta>0$ there is $N_1=N_1(\eta)\in\mathbb{N}$ such that for any $N\ge N_1$ we have $\|(I-\Pi_N)B^{\T}\|_{L^2[0,l]}<\eta$, hence
\begin{equation}\label{A18}
\int\limits_0^lu_N(z,t)(\Pi_N-I)B^{\T}(z)\,dzx_N(t)\le \eta\|x_N(t)\|\|u_N(\cdot,t)\|_{L^2[0,l]}
\end{equation}
\begin{equation}\label{A18bis}
x_N^{\T}(t)\int\limits_0^l(P_{12}(z)(\Pi_N-I)B^{\T}(z)+(\Pi_N-I)B(z)P_{12}^{\T}(z))\,dzx_N(t)\\
\le 2\eta\|P_{12}\|_{L^2[0,l]}\|x_N(t)\|^2.
\end{equation}
By the Cauchy-Bunyakovskiy inequality, taking  \eqref{3****} into account we have
\begin{equation*}\label{A19}
	\gathered
	X^{\T}(x_N(t))\int\limits_0^lP_{12}(z)u_N(z,t)\,dz\le\delta_2\|x_N(t)\|^{2q-1}\int\limits_0^l\|P_{12}(z)\||u_N(z,t)|\,dz\\
	\le
	\delta_2\|P_{12}\|_{L^2[0,l]}\|x_N(t)\|^{2q-1}\|u_N(\cdot,t)\|_{L^2[0,l]}.
	\endgathered
\end{equation*}
By the Young's inequality with $p_1=\frac{2q}{2q-1}$, $p_2=2q$ and $\tau>0$ we obtain
\begin{equation*}
	\gathered
	\|x_N(t)\|^{2q-1}\|u_N(\cdot,t)\|_{L^2[0,l]}\le\frac{2q-1}{2q}\tau^{2q/(2q-1)}\|x_N(t)\|^{2q}+\frac{\tau^{-2q}}{2q}\|u_N(\cdot,t)\|^{2q}_{L^2[0,l]},
	\endgathered
\end{equation*}
and taking \eqref{W3bis} into account we further estimate
\begin{equation}\label{A20}
\gathered
X^{\T}(x_N(t))\int\limits_0^lP_{12}(z)u_N(z,t)\,dz\le\frac{\delta_2\|P_{12}\|_{L^2[0,l]}(2q-1)}{2q}\tau^{2q/(2q-1)}\|x_N(t)\|^{2q}\\
+
\frac{\delta_2\|P_{12}\|_{L^2[0,l]}l^{q-1}}{2q}\tau^{-2q}\|u_N(\cdot,t)\|_{L^{2q}[0,l]}^{2q}.
\endgathered
\end{equation}
By the self-adjointness of $\Pi_N$ we derive that
\begin{equation*}
	\gathered
	x_N^{\T}(t)\int\limits_0^lP_{12}(z)\Pi_Nf(u_N(z,t))\,dz=(x_N^{\T}(t)P_{12},\Pi_Nf(u_N))_{L^2[0,l]}\\=
	(\Pi_Nx_N^{\T}(t)P_{12},f(u_N))_{L^2[0,l]}\\=
	(x_N^{\T}(t)\Pi_NP_{12},f(u_N))_{L^2[0,l]}=x_N^{\T}(t)\int\limits_0^l\Pi_NP_{12}(z)f(u_N(z,t))\,dz.
	\endgathered
\end{equation*}
Hence, taking \eqref{3*} and \eqref{3***} into account and applying the Cauchy-Bunyakovskiy inequality it follows that
\begin{equation*}
	\gathered
	\Big|x_N^{\T}(t)\int\limits_0^l\Pi_NP_{12}(z)f(u_N(z,t))\,dz\Big|\le
	\|x_N(t)\|\int\limits_0^l\|\Pi_NP_{12}(z)\||f(u_N(z,t))|\,dz\\
	\le
	\|x_N(t)\|\int\limits_0^l\|\Pi_NP_{12}(z)\||f_0(u_N(z,t))|\,dz+\|x_N(t)\|\int\limits_0^l\|\Pi_NP_{12}(z)\||f_1(u_N(z,t))|\,dz\\
	\le L\|x_N(t)\|\|\Pi_N P_{12}\|_{L^2[0,l]}\|u_N(\cdot,t)\|_{L^2[0,l]}+\zeta\|x_N(t)\|\int\limits_0^l\|\Pi_NP_{12}(z)\||u_N(z,t)|^{2q-1}\,dz\\
	\le L\|x_N(t)\|\|P_{12}\|_{L^2[0,l]}\|u_N(\cdot,t)\|_{L^2[0,l]}+\zeta\int\limits_0^l\|\Pi_NP_{12}(z)\|\|x_N(t)\||u_N(z,t)|^{2q-1}\,dz.
	\endgathered
\end{equation*}
To estimate $\|x_N(t)\||u_N(z,t)|^{2q-1}$ we apply the Young's inequality \eqref{Yang} with $p_1=2q$, $p_2=\frac{2q}{2q-1}$,  $\tau>0$
\begin{equation}\label{A20**}
\gathered
\|x_N(t)\||u_N(z,t)|^{2q-1}\le\frac{\tau^{2q}}{2q}\|x_N(t)\|^{2q}+\frac{2q-1}{2q}\tau^{-2q/(2q-1)}|u_N(z,t)|^{2q}.
\endgathered
\end{equation}
Obviously  $\Pi_NP_{12}\in (L^1[0,l])^n$ and
$\|(I-\Pi_N)P_{12}\|_{L^1[0,l]}\le\sqrt{l}\|(I-\Pi_N)P_{12}\|_{L^2[0,l]}\to 0$ for $N\to\infty$. Since $P_{12}\in (L^2[0,l])^n$, for any $\eta>0$ there is $N_2=N_2(\eta)\in\mathbb{N}$, such that $\|\Pi_NP_{12}\|_{L^1[0,l]}\le \|P_{12}\|_{L^1[0,l]}+\eta$.
Also we have $\Pi_NP_{12}\in (L^{\infty}[0,l])^n$ and we can show that $\|(I-\Pi_N)P_{12}\|_{L^{\infty}[0,l]}\to 0$ for $N\to\infty$.
Indeed, let $P_{12}(z)=(P_{12}^{(1)}(z),\dots,P_{12}^{(n)})^{\T}$, then for $k=1,\dots,n$ holds:
\begin{equation*}
	\gathered
	(I-\Pi_N)P_{12}^{(k)}(z)=\sum\limits_{p=N+1}^{\infty}(P_{12}^{(k)},e_p)_{L^2[0,l]}e_p(z),\\
	(P_{12}^{(k)},e_p)_{L^2[0,l]}=\sqrt{\frac{2}{l}}\int\limits_0^lP_{12}^{(k)}(z)\sin\frac{\pi pz}{l}\,dz.
	\endgathered
\end{equation*}
Taking into account that $P_{12}\in (H^2_0[0,l])^n$ and integrating two times by parts we get
\begin{equation*}
	\gathered
	(P_{12}^{(k)},e_p)_{L^2[0,l]}=\sqrt{\frac{2}{l}}\frac{l}{\pi p}\int\limits_0^l\partial_zP_{12}^{(k)}(z)\cos\frac{\pi pz}{l}\,dz\\
	=-\sqrt{\frac{2}{l}}\frac{l^2}{\pi^2 p^2}\int\limits_0^l\partial_{zz}P_{12}^{(k)}(z)\sin\frac{\pi pz}{l}\,dz.
	\endgathered
\end{equation*}
From which follows
\begin{equation*}
	\gathered
	|(P_{12}^{(k)},e_p)_{L^2[0,l]}|\le\frac{l^2}{\pi^2 p^2}\|\partial_{zz}P_{12}^{(k)}\|_{L^2[0,l]}.
	\endgathered
\end{equation*}
Hence
\begin{equation*}
	\gathered
	|(I-\Pi_N)P_{12}^{(k)}(z)|\le\frac{\sqrt{2}l^{3/2}\|\partial_{zz}P_{12}^{(k)}\|_{L^2[0,l]}}{\pi^2}\sum\limits_{p=N+1}^{\infty}\frac{1}{p^2}.
	\endgathered
\end{equation*}
so that
\begin{equation*}
	\gathered
	\|(I-\Pi_N)P_{12}\|_{L^{\infty}[0,l]}\le\frac{\sqrt{2}l^{3/2}\|\partial_{zz}P_{12}\|_{L^2[0,l]}}{\pi^2}\sum\limits_{p=N+1}^{\infty}\frac{1}{p^2}\to 0\quad\text{for}\quad N\to\infty.
	\endgathered
\end{equation*}
We conclude that for any $\eta>0$ there exists $N_3=N_3(\eta)\in\mathbb{N}$ such that for all $N\ge N_3$ we have $\|\Pi_NP_{12}\|_{L^{\infty}[0,l]}\le\|P_{12}\|_{L^{\infty}[0,l]}+\eta$.

With help of \eqref{A20**} we can estimate the integral
\begin{equation*}
	\gathered
	\int\limits_0^l\|\Pi_NP_{12}(z)\|\|x_N(t)\||u_N(z,t)|^{2q-1}\,dz\\
	\le \frac{\tau^{2q}}{2q}\|\Pi_NP_{12}\|_{L^1[0,l]}\|x_N(t)\|^{2q}+\frac{2q-1}{2q}\tau^{-2q/(2q-1)}\|\Pi_NP_{12}\|_{L^{\infty}[0,l]}\|u_N(\cdot,t)\|^{2q}_{L^{2q}[0,l]}\\
	\le \frac{\tau^{2q}}{2q}(\|P_{12}\|_{L^1[0,l]}+\eta)\|x_N(t)\|^{2q}\\
	+\frac{2q-1}{2q}\tau^{-2q/(2q-1)}(\|P_{12}\|_{L^{\infty}[0,l]}+\eta)\|u_N(\cdot,t)\|^{2q}_{L^{2q}[0,l]}.
	\endgathered
\end{equation*}
Finally we obtain
\begin{equation}\label{A21}
\gathered
\Big|x_N^{\T}(t)\int\limits_0^l\Pi_NP_{12}(z)f(u_N(z,t))\,dz\Big|
\le L\|x_N(t)\|\|P_{12}\|_{L^2[0,l]}\|u_N(\cdot,t)\|_{L^2[0,l]}\\+
\zeta\frac{\tau^{2q}}{2q}(\|P_{12}\|_{L^1[0,l]}+\eta)\|x_N(t)\|^{2q}\\
+\zeta\frac{2q-1}{2q}\tau^{-2q/(2q-1)}(\|P_{12}\|_{L^{\infty}[0,l]}+\eta)\|u_N(\cdot,t)\|^{2q}_{L^{2q}[0,l]}.
\endgathered
\end{equation}
By the condition 2) of Theorem \ref{main-theorem} it follows that
\begin{equation}\label{A22}
x_N^{\T}(t)PX(x_N(t))\le-\delta_1\|x_N(t)\|^{2q}.
\end{equation}
Applying the Cauchy-Bunyakovskiy inequality we have
\begin{equation}\label{A23}
\gathered
\int\limits_0^lu_N(z,t)\Pi_NH_t(z,t)\,dz\le\|H_t\|_{L^2[0,l]}\|u_N(\cdot,t)\|_{L^2[0,l]},
\endgathered
\end{equation}
\begin{equation}\label{A24}
\gathered
\Big|\int\limits_0^lH(z,t)D^{\T}(z)\,dz\int\limits_0^lP_{12}(z)u_N(z,t)\,dz\Big|\\
\le\|D\|_{L^2[0,l]}\|H\|_{L^2[0,l]}\|P_{12}\|_{L^2[0,l]}\|u_N(\cdot,t)\|_{L^2[0,l]},
\endgathered
\end{equation}
\begin{equation}\label{A25}
\gathered
\Big|x_N^{\T}(t)\int\limits_0^lP_{12}(z)\Pi_NH_t(z,t)\,dz\Big|\le\|H_t\|_{L^2[0,l]}\|P_{12}\|_{L^2[0,l]}\|x_N(t)\|,
\endgathered
\end{equation}
\begin{equation}\label{A26}
\gathered
\Big|x_N^{\T}(t)\int\limits_0^lPD(z)H(z,t)\,dz\Big|\le\|P\|\|D\|_{L^2[0,l]}\|H\|_{L^2[0,l]}\|x_N(t)\|.
\endgathered
\end{equation}
From the self-adjointness of $\Pi_N$  follows
\begin{equation*}
	\gathered
	\int\limits_0^lu_N(z,t)\Pi_N(f(u_N(z,t)+H(z,t))-f(u_N(z,t)))\,dz\\
	=\int\limits_0^lu_N(z,t)(f(u_N(z,t)+H(z,t))-f(u_N(z,t)))\,dz\\
	=
	\int\limits_0^lu_N(z,t)(f_0(u_N(z,t)+H(z,t))-f_0(u_N(z,t)))\,dz\\+
	\int\limits_0^lu_N(z,t)(f_1(u_N(z,t)+H(z,t))-f_1(u_N(z,t)))\,dz\\
	\le L\int\limits_0^l|u_N(z,t)||H(z,t)|\,dz\\
	+\int\limits_0^l|u_N(z,t)||f_1(u_N(z,t)+H(z,t))-f_1(u_N(z,t))|\,dz
	\endgathered
\end{equation*}
Using \eqref{3***} and the mean value theorem there is some $\vartheta\in(0,1)$ such that
\begin{equation}\label{A26bis}
\gathered
|f_1(u_N(z,t)+H(z,t))-f_1(u_N(z,t))|\le |f_1^{\prime}((1-\vartheta)u_N(z,t)+\vartheta H(z,t))||H(z,t)|\\
\le c_0(1+((1-\vartheta)|u_N(z,t)|+\vartheta|H(z,t)|)^{2q-2})|H(z,t)|\\
\le c_0(|H(z,t)|+|u_N(z,t)|^{2q-2}|H(z,t)|+|H(z,t)|^{2q-1}).
\endgathered
\end{equation}
By the inequalities of Cauchy-Bunyakovsky and Young we derive that
\begin{equation}\label{A26*}
\gathered
\int\limits_0^l|u_N(z,t)||f_1(u_N(z,t)+H(z,t))-f_1(u_N(z,t))|\,dz\\
\le c_0\int\limits_0^l|u_N(z,t)|(|H(z,t)|+|H(z,t)|^{2q-1})\,dz+c_0\int\limits_0^l|u_N(z,t)|^{2q-1}|H(z,t)|\,dz\\
\le \frac{c_0\epsilon}{2}\|u_N(\cdot,t)\|_{L^2[0,l]}^2+\frac{c_0\epsilon^{-1}}{2}\||H(\cdot,t)|+|H(\cdot,t)|^{2q-1}\|_{L^2[0,l]}^2\\
+
\frac{c_0\epsilon^{2q/(2q-1)}(2q-1)}{2q}\|u_N(\cdot,t)\|_{L^{2q}[0,l]}^{2q}+\frac{c_0\epsilon^{-2q}}{2q}\|H(\cdot,t)\|_{L^{2q}[0,l]}^{2q}
\endgathered
\end{equation}
and finally we conclude
\begin{equation}\label{A27}
\gathered
\int\limits_0^lu_N(z,t)\Pi_N(f(u_N(z,t)+H(z,t))-f(u_N(z,t)))\,dz\\
\le \frac{L\epsilon}{2}\|u_N(\cdot,t)\|_{L^2[0,l]}^2+\frac{L\epsilon^{-1}}{2}\|H(\cdot,t)\|_{L^2[0,l]}^2\\
+\frac{c_0\epsilon}{2}\|u_N(\cdot,t)\|_{L^2[0,l]}^2+\frac{c_0\epsilon^{-1}}{2}\||H(\cdot,t)|+|H(\cdot,t)|^{2q-1}\|_{L^2[0,l]}^2\\
+
\frac{c_0\epsilon^{2q/(2q-1)}(2q-1)}{2q}\|u_N(\cdot,t)\|_{L^{2q}[0,l]}^{2q}+\frac{c_0\epsilon^{-2q}}{2q}\|H(\cdot,t)\|_{L^{2q}[0,l]}^{2q}.
\endgathered
\end{equation}

Again we use that  $\Pi_N$ is self-adjoint to derive
\begin{equation*}
	\gathered
	x_N^{\T}(t)\int\limits_0^lP_{12}(z)\Pi_N(f(u_N(z,t)+H(z,t))-f(u_N(z,t)))\,dz\\=
	x_N^{\T}(t)\int\limits_0^l\Pi_NP_{12}(z)(f(u_N(z,t)+H(z,t))-f(u_N(z,t)))\,dz\\
	\le\|x_N(t)\|\int\limits_0^l\|\Pi_NP_{12}(z)\||f(u_N(z,t)+H(z,t))-f(u_N(z,t))|\,dz
	\endgathered
\end{equation*}
and taking  \eqref{A26bis} into account
\begin{equation*}
	\gathered
	\|x_N(t)\|\int\limits_0^l\|\Pi_NP_{12}(z)\||f_1(u_N(z,t)+H(z,t))-f_1(u_N(z,t))|\,dz\\
	\le c_0\|x_N(t)\|\int\limits_0^l\|\Pi_NP_{12}(z)\|(|H(z,t)|+|u_N(z,t)|^{2q-2}|H(z,t)|+|H(z,t)|^{2q-1})\,dz\\
	\le c_0\|P_{12}\|_{L^2[0,l]}\||H(\cdot,t)|+|H(\cdot,t)|^{2q-1}\|_{L^2[0,l]}\|x_N(t)\|\\+
	c_0\|x_N(t)\|\int\limits_0^l\|\Pi_NP_{12}(z)\||u_N(z,t)|^{2q-2}|H(z,t)|\,dz
	\endgathered
\end{equation*}
By the Young's inequality \eqref{Yang} with $p_1=(2q-1)/(2q-2)$, $p_2=2q-1$ we see that
\begin{equation*}
	\|x_N(t)\||u_N(z,t)|^{2q-2}\le \frac{2q-2}{2q-1}|u_N(z,t)|^{2q-1}+\frac{1}{2q-1}\|x_N(t)\|^{2q-1}
\end{equation*}
Hence
\begin{equation*}
	\gathered
	\|x_N(t)\|\int\limits_0^l\|\Pi_NP_{12}(z)\||u_N(z,t)|^{2q-2}|H(z,t)|\,dz\\
	\le\frac{2q-2}{2q-1}\int\limits_0^l\|\Pi_NP_{12}(z)\||H(z,t)||u_N(z,t)|^{2q-1}\,dz\\
	+\frac{1}{2q-1}\|x_N(t)\|^{2q-1}
	\int\limits_0^l\|\Pi_NP_{12}(z)\||H(z,t)|\,dz
	\endgathered
\end{equation*}
Applying the Young's inequality two times we obtain
\begin{equation*}
	\gathered
	\|x_N(t)\|^{2q-1}
	\int\limits_0^l\|\Pi_NP_{12}(z)\||H(z,t)|\,dz\\
	\le
	\frac{2q-1}{2q}\epsilon^{2q/(2q-1)}\|x_N(t)\|^{2q}+\frac{1}{2q}\epsilon^{-2q}\Big(\int\limits_0^l\|\Pi_NP_{12}(z)\||H(z,t)|\,dz\Big)^{2q}\\
	\le \frac{2q-1}{2q}\epsilon^{2q/(2q-1)}\|x_N(t)\|^{2q}+\frac{1}{2q}\epsilon^{-2q}\|P_{12}\|_{L^2[0,l]}^{q}\|H(\cdot,t)\|_{L^2[0,l]}^{2q}.
	\endgathered
\end{equation*}
\begin{equation*}
	\gathered
	\|\Pi_NP_{12}(z)\||H(z,t)||u_N(z,t)|^{2q-1}\\
	\le\|\Pi_NP_{12}\|_{L^{\infty}[0,l]}(\frac{2q-1}{2q}\epsilon^{2q/(2q-1)}|u_N(z,t)|^{2q}
	+\frac{1}{2q}\epsilon^{-2q}\,|H(z,t)|^{2q})
	\endgathered
\end{equation*}
and finally we get an estimate for the integral
\begin{equation*}
	\gathered
	\int\limits_0^l\|\Pi_NP_{12}(z)\||H(z,t)||u_N(z,t)|^{2q-1}\,dz\\
	\le\|\Pi_NP_{12}\|_{L^{\infty}[0,l]}\frac{2q-1}{2q}\epsilon^{2q/(2q-1)}\|u_N(\cdot,t)\|^{2q}_{L^{2q}[0,l]}\\
	+\|\Pi_NP_{12}\|_{L^{\infty}[0,l]}\frac{1}{2q}\epsilon^{-2q}\|H(\cdot,t)\|^{2q}_{L^{2q}[0,l]}.
	\endgathered
\end{equation*}
In the same way we can derive that
\begin{equation*}
	\gathered
	x_N^{\T}(t)\int\limits_0^lP_{12}(z)\Pi_N(f_0(u_N(z,t)+H(z,t))-f_0(u_N(z,t)))\,dz\\
	\le
	\frac{L}{2}\|P_{12}\|_{L^2[0,l]}\|x_N(t)\|^2+\frac{L}{2}\|P_{12}\|_{L^2[0,l]}.
	\endgathered
\end{equation*}
Let $N\ge N_3(\eta)$, then $\|\Pi_NP_{12}\|_{L^{\infty}[0,l]}\le \|P_{12}\|_{L^{\infty}[0,l]}+\eta$. Hence
\begin{equation}\label{A28}
\gathered
x_N^{\T}(t)\int\limits_0^lP_{12}(z)\Pi_N(f(u_N(z,t)+H(z,t))-f(u_N(z,t)))\,dz\\
\le \Big(\frac{c_0\|P_{12}\|_{L^2[0,l]}\||H(\cdot,t)|+|H(\cdot,t)|^{2q-1}\|_{L^2[0,l]}\epsilon}{2}+\frac{L}{2}\|P_{12}\|_{L^2[0,l]}\Big)\|x_N(t)\|^2\\
+\frac{c_0\|P_{12}\|_{L^2[0,l]}\||H(\cdot,t)|+|H(\cdot,t)|^{2q-1}\|_{L^2[0,l]}\epsilon^{-1}}{2}\\
+\frac{c_0(q-1)}{q}(\|P_{12}\|_{L^{\infty}[0,l]}+\eta)\epsilon^{2q/(2q-1)}\|u_N(\cdot,t)\|_{L^{2q}[0,l]}^{2q}+
\frac{c_0}{2q}\epsilon^{2q/(2q-1)}\|x_N(t)\|^{2q}\\
+\frac{\epsilon^{-2q}c_0}{q(2q-1)}\Big((q-1)(\|P_{12}\|_{L^{\infty}[0,l]}+\eta)\|H\|_{L^{2q}[0,l]}^{2q}\\
+\frac{\|P_{12}\|_{L^2[0,l]}^{2q}}{2}\|H\|_{L^2[0,l]}^{2q}\Big)+\frac{L}{2}\|P_{12}\|_{L^2[0,l]}.
\endgathered
\end{equation}

Collecting the estimations \eqref{A15*}---\eqref{A28} to the estimate for the time derivative of  $V(u_N(\cdot,t),x_N(t))$ along solutions to the Galerkin system \eqref{A4}, which holds for all  $N\ge N^{*}(\eta)=\max(N_1(\eta),N_2(\eta),N_3(\eta))$:
\begin{equation}\label{A29}
\gathered
\dot V(u_N(\cdot,t),x_N)\le -x_N^{\T}(t)\Omega(\eta,\epsilon)x_N(t)-\omega(\eta,\epsilon)\|u_N(\cdot,t)\|^2_{L^2[0,l]}\\
+(2L\|P_{12}\|_{L^2[0,l]}+\xi(\eta,\epsilon))\|x_N(t)\|\|u_N(\cdot,t)\|_{L^2[0,l]}-\\
-\wt H_1(\eta,\epsilon,\tau)\|u_N(\cdot,t)\|_{L^{2q}[0,l]}^{2q}-\wt H_2(\eta,\epsilon,\tau)\|x_N(t)\|^{2q}+\wt H_3(\epsilon,\eta)
\endgathered
\end{equation}
where $\Omega(\eta,\epsilon)\to\Omega$, $\omega(\eta,\epsilon)\to\omega$, $\xi(\eta,\epsilon)\to 0$ for $(\eta,\epsilon)\to (0,0)$,
\begin{equation*}
	\wt H_1(\eta,\epsilon,\tau)\to 2\al-\Big(\frac{\delta_2\|P_{12}\|_{L^2[0,l]}l^{q-1}}{q}\tau^{-2q}+\frac{\zeta(2q-1)\|P_{12}\|_{L^{\infty}[0,l]}}{q}\tau^{-2q/(2q-1)}\Big),
	\end{equation*}
\begin{equation*}
	\gathered
	\wt H_2(\eta,\epsilon,\tau)\to 2\delta_1-\Big(\frac{\zeta\|P_{12}\|_{L^1[0,l]}}{q}\tau^{2q}+\frac{(2q-1)\delta_2\|P_{12}\|_{L^2[0,l]}}{q}\tau^{2q/(2q-1)}\Big),
	\endgathered
\end{equation*}
for $(\eta,\epsilon)\to (0,0)$.

The condition 5) of Theorem \ref{main-theorem} implies the existence of  $\tau>0$ such that the inequalities  \eqref{tauINEQ} hold.
We fix such $\tau>0$ and note that by the positive definiteness of the matrix $\Xi$ it follows that one can choose $\eta$ and $\epsilon$ small enough, so that the quadratic form  \eqref{A29} becomes negative definite and
$\wt H_1(\eta,\epsilon,\tau)>0$, $\wt H_2(\eta,\epsilon,\tau)>0$. In this cas the inequality  \eqref{A29} for $N\ge N^*$ can be written as
\begin{equation}\label{A30}
\dot V(u_N(\cdot,t),x_N(t))\le -\gamma_0 V(u_N(\cdot,t),x_N(t))
-\gamma_1(\|u_N(\cdot,t)\|_{L^{2q}[0,l]}^{2q}+\|x_N(t)\|^{2q})+\gamma_2
\end{equation}
where $\gamma_i>0$, $i=1,2,3$. From \eqref{A30} follows
\begin{equation*}
	\gathered
	\lambda_{\min}(\Pi_1)(\|x_N(t)\|^2+\|u_N(\cdot,t)\|^2_{L^2[0,l]})\le V(u_N(\cdot,t),x_N(t))\\
	\le e^{-\gamma_0t} V(\vph_N,x(0))
	+\frac{\gamma_2}{\gamma_0}\le e^{-\gamma_0t}\lambda_{\max}(\Pi_2)(\|x_0\|^2+\|\vph\|^2_{L^2[0,l]})
	+\frac{\gamma_2}{\gamma_0}.
	\endgathered
\end{equation*}
From which the a priori estimate \eqref{A30bis} follows.

Integrating  \eqref{A30} from $0$ to $T$ we obtain the a priori estimate \eqref{A30*}.
\end{proof}
\begin{lemma}\label{lemmaA2}
There exists $N^*$ such that for all $N\ge N^*$ and $t\in\mathbb{R}_+$ for the solutions $(u_N,x_N)$ to the ODE system \eqref{A4} the following is true
\begin{enumerate}
\item[(i)] The mapping $N_f\,:L^{2q}[0,l]\to L^{2q/(2q-1)}[0,l]$ defined by $N_f: u\mapsto f(u(x))$ is bounded.
\item[(ii)] The linear operator  $\partial_{zz}\,:\,H_0^1[0,l]\to H^{-1}[0,l]$ is bounded.
\item[(iii)] For any $T>0$ the mapping $\partial_zu_N\in L^2([0,T],L^2[0,l])$ and
\begin{equation*}
	\|u_N(\cdot,T)\|_{L^2[0,l]}^2+a^2\|\partial_z u_N\|_{L^2([0,T],L^2[0,l])}^2\le\|\vph\|^2_{L^2[0,l]}+c_2 T.
\end{equation*}
\item[(iv)] For any $s\ge 1$ the sequence $\{\partial_t u_N\}_{N^*}^\infty$ is bounded in $L^{2q/(2q-1)}([0,T],(H_0^s)^*)$.
\end{enumerate}
\end{lemma}
\begin{proof}
(i) Note that if $u\in L^{2q}[0,l]$, then  $f(u)\in L^{2q/(2q-1)}[0,l]$ and the mapping
$f\,:\,L^{2q}[0,l]\to L^{2q/(2q-1)}[0,l]$ is bounded, because by the H\"older's inequality there is some constant $c^*>0$ such that
\begin{equation*}
	\|f(u)\|_{L^{2q/(2q-1)}[0,l]}\le c^*\|u\|_{L^{2q}[0,l]}.
\end{equation*}
(ii) The linear operator $\partial_{zz}$ is bounded as a mapping from $H^1[0,l]$ to $H^{-1}[0,l]$. Indeed, for any $\vph\in C^{\infty}_0[0,l]$ the following inequality holds
\begin{equation*}
	\gathered
	\sup\limits_{\vph\in C^{\infty}_0[0,l],\,\vph\ne 0}\frac{|\langle\partial_{zz}u,\vph\rangle|}{\|\vph\|_{H^1[0,l]}}
	= \sup\limits_{\vph\in C^{\infty}_0[0,l],\,\vph\ne 0}\frac{|(\partial_{z}u,\partial_{z}\vph)_{L^2[0,l]}|}{\|\vph\|_{H^1[0,l]}}\\
	\le \sup\limits_{\vph\in C^{\infty}_0[0,l],\,\vph\ne 0}\frac{\|\partial_{z}u\|_{L^2[0,l]}\|\partial_{z}\vph\|_{L^2[0,l]}}{\|\vph\|_{H^1[0,l]}}
	\le \|\partial_{z}u\|_{L^2[0,l]}\le\|u\|_{H^1[0,l]}.
	\endgathered
\end{equation*}
Since $C^{\infty}_0[0,l]$ is dence in $H_0^1[0,l]$, we obtain $\|\partial_{zz}u\|_{H^{-1}[0,l]}\le\|u\|_{H^1[0,l]}$.

(iii) Consider the Lyapunov function $V_1(u_N(\cdot,t))=\|u_N(\cdot,t)\|_{L^2[0,l]}^2$. Its time derivative with respect to the first equation of the Galerkin system \eqref{A4} can be estimated as
\begin{equation*}
	\gathered
	\dot V_1(u_N(\cdot,t))=\langle u_N(\cdot,t),a^2\partial_{zz}u_N(\cdot,t)\rangle+\langle u_N(\cdot,t),\Pi_Nf(u_N(\cdot,t))\rangle\\
	+\langle u_N(\cdot,t),\Pi_N(f(u_N(\cdot,t)+H)-f(u_N(\cdot,t)))\rangle\\
	+\langle u_N(\cdot,t),B^{\T}(z)x_N(t)\rangle-
	\langle u_N(\cdot,t),H_t(\cdot,t)\rangle\\
	\le-a^2\|\partial_zu_N(\cdot,t)\|_{L^2[0,l]}^2+(\sigma+L)\|u_N(\cdot,t)\|_{L^2[0,l]}^2-\alpha\|u_N(\cdot,t)\|_{L^{2q}[0,l]}^{2q}\\
	+\int\limits_0^l |u_N(z,t)||f_1(u_N(z,t)+H(z,t))-f_1(u_N(z,t))|\,dz\\
	+\langle u_N(\cdot,t),B^{\T}(z)x_N(t)\rangle-
	\langle u_N(\cdot,t),H_t(\cdot,t)\rangle
	\endgathered
\end{equation*}
Taking  \eqref{A26*}  into account we get
\begin{equation*}
	\gathered
	\dot V_1(u_N(\cdot,t))\le
	-a^2\|\partial_zu_N(\cdot,t)\|_{L^2[0,l]}^2+(\sigma+L+\frac{c_0\epsilon}{2})\|u_N(\cdot,t)\|_{L^2[0,l]}^2\\
	+(-\alpha+\frac{c_0\epsilon^{2q/(2q-1)}(2q-1)}{2q})\|u_N(\cdot,t)\|_{L^{2q}[0,l]}^{2q}
	+c_2^{\prime}+c_3^{\prime}\|u_N(\cdot,t)\|_{L^2[0,l]},
	\endgathered
\end{equation*}
for some $c_2^{\prime}>0$, $c_3^{\prime}>0$.

Choosing $\epsilon$ small enough and using the boundedness of  $\{u_N\}_{N^*}^{\infty}$ in  $L^2[0,l]$
proved in Lemma \ref{lemmaA1} we obtain for some $c_2>0$
\begin{equation*}
	\dot V_1(u_N(\cdot,t))\le
	-a^2\|\partial_zu_N(\cdot,t)\|_{L^2[0,l]}^2+c_2.
\end{equation*}
Integrating this inequality from $0$ to $T$ we obtain (iii).

(iv) Note that the sequence $\{u_N\}_{N^*}^\infty$ is bounded in $L^{2}([0,T],H_0^1[0,l])\cap L^{2q/(2q-1)}([0,T],L^{2q/(2q-1)}[0,l])$,
the mappings $\partial_{zz}u_N$ and $f$ are bounded in $L^2([0,T],H^{-1}[0,l])$ and  $L^{2q/(2q-1)}([0,T],L^{2q/(2q-1)}[0,l])$ respectively. Hence $\partial_{zz}u_N$ and $f$ are bounded in  $L^{2q/(2q-1)}([0,T],(H_0^s)^*)$ for any $s>1$, this means that
$\{\partial_t u_N\}_{N^*}^\infty$ is bounded in $L^{2q/(2q-1)}([0,T],(H_0^s)^*)$.
\end{proof}
\begin{proof}[Proof of Theorem \ref{main-theorem}]
Since the space $L^{2}([0,T],L^2[0,l])$ is a Hilbert space, its bounded subsets are compact in weak topology.
By the Lemma \ref{lemmaA1}  we have the boundedness of the sequence $\{u_N\}_{N=N^*}^{\infty}$ in $L^2([0,T],L^2[0,l])$, hence there is a subsequence which converges weakly to the function $u\in L^{2}([0,T],L^2[0,l])$.
The space $L^{2q}([0,T],L^{2q}[0,l])$ is reflexive as well, hence its bounded subsets are weakly compact. Hence there is a subsequence in the latter one that converges weakly in $L^{2q}([0,T],L^{2q}[0,l])$ and its limit is again $u$.

The time derivatives of the elements  $t$ of this subsequence is bounded in  $L^{2q/(2q-1)}([0,T],(H_0^s)^*)$ and by its reflexivity there is a subsequence  $\partial_t u_N$ that converges weakly in $L^{2q/(2q-1)}([0,T],H^{-s})$ to some function $v$ and we have
$v=\partial_t u$.

The mapping $\,u\mapsto f(u+H(\cdot,t))$ is bounded in  $L^{2q/(2q-1)}([0,T],L^{2q/(2q-1)}[0,l])$, hence we can chose a subsequence $u_N$ from the latter one , such that $f(u_N(\cdot,t)+H(\cdot,t))\to\mu$   weakly in $L^{2q/(2q-1)}([0,T],L^{2q/(2q-1)}[0,l])$. By Theorem \ref{teorema-o-kompakt} applied to $\mathcal X_0=(H_0^s)^*$, $\mathcal X=L^2[0,l]$, $\mathcal X_1=L^{2q}[0,l]$, $p_0=2$, $p_1=2q/(2q-1)$ and by Lemmas \ref{lemmaA1},\ref{lemmaA2}, it follows that the chosen subsequence from  $\{u_N\}_{N=N^*}^{\infty}$ converges to $u$ in $L^2([0,T],L^2[0,l])$. Taking a further subsequence (if necessary) by the Riesz theorem follows that there is a subsequence of functions $\{u_N(\cdot,t)\}_{N=N^*}^{\infty}$ that converges almost everywhere on $[0,T]$ to $u(\cdot,t)$ in
$L^2[0,l]$.

From the sequence of functions $x_N\in C([0,T],\mathbb{R}^n)$ we select the subsequence $\{x_N\}_{N^*}^{\infty}$ corresponding to the last chosen subsequence of $\{u_N\}_{N=N^*}^{\infty}$.
From \eqref{A30bis} and the second equation of \eqref{A4} it follows that there exists $c_1>0$ such that for all $N\ge N^*$ the time derivatives of $x_N$ are unifromly bounded: $\sup\limits_{t\in\mathbb{R}_+}|\dot x_N(t)|\le c_1$. By the theorem of Arzela-Ascoli follows that there is a subsequence of $\{x_N\}_{N^*}^{\infty}$ which converges (in $C([0,T],\mathbb{R}^n)$). The final subsequence obtained from $\{(u_N,x_N)\}_{N=N^*}^{\infty}$ in this way we denote by
$(u_m,x_m)$. Let $x(t)=\lim_{m\to\infty}x_m(t)$ in $C([0,T],\mathbb{R}^n)$.

The sequence $\{u_m\}_{m=m_0}^{\infty}$, $m_0\ge N^*$ is a subsequence of $\{u_N\}_{N^*}^{\infty}$ and possess all the properties mentioned above. Since $\partial_t u_m\to\partial_t u$ weakly in $L^{2q/(2q-1)}([0,T],H^{-s})$ and $u_m(\cdot,0)\to \vph$ in $L^{2}[0,l]$, then
\begin{equation*}
	u_m(\cdot,t)=\int\limits_0^{t}\partial_t u_m(\cdot,\tau)\,d\tau+u_m(\cdot,0)\to \int\limits_0^{t}\partial_t u(\cdot,\tau)\,d\tau+u(\cdot,0)
	\end{equation*}
for each $t\in[0,T]$ weakly in $H^{-s}$ for $m\to\infty$. Then it follows that $u(\cdot,t)\to u_0$ weakly in $H^{-s}$ fro $t\to 0+$. This proves that the initial condition is satisfied.

From \eqref{A4} follows
\begin{equation}\label{A31}
\gathered
\langle\partial_t u_m(\cdot,t),e_j\rangle=a^2\langle\partial_{zz}u_{m}(\cdot,t),e_j\rangle+\langle f_0(u_m(\cdot,t)+H(\cdot,t)),e_j\rangle\\
+\langle f_1(u_m(\cdot,t)+H(\cdot,t)),e_j\rangle+\langle B^{\T}(\cdot)x_m(t),e_j\rangle-\langle H_t(\cdot,t),e_j\rangle.
\endgathered
\end{equation}
By the weak convergence $u_m\to u$ for $m\to\infty$ and definition of the generalized derivative it follows that
\begin{equation*}
	\gathered
	\langle\partial_{zz}u_m(\cdot,t),e_j\rangle=\langle u_m(\cdot,t),\partial_{zz} e_j\rangle
	\to \langle u(\cdot,t),\partial_{zz} e_j\rangle=\langle\partial_{zz} u(\cdot,t), e_j\rangle
	\endgathered
\end{equation*}
for $m\to\infty$.

Recall that $f_0$ as a part of $f$ is globally Lipschitz, so that
\begin{equation}\label{A37*}
\gathered
\langle f_0(u_m(\cdot,t)+H(\cdot,t))-f_0(u(\cdot,t)+H(\cdot,t)),e_j\rangle\\
=\int\limits_0^l (f_0(u_m(z,t)+H(z,t))-f_0(u(z,t)+H(z,t))e_j(z)\,dz\\
\le L\int\limits_0^l|u_m(z,t)-u(z,t)||e_j(z)|\,dz\le L\|u_m(\cdot,t)-u(\cdot,t)\|_{L^2[0,l]}\to 0,\quad{for}\quad m\to\infty.
\endgathered
\end{equation}
By the convergence of $B^{\T}(z)x_m(t)\to B^{\T}(z)x(t)$ in $L^2[0,l]$ for $m\to\infty$, we obtain from   \eqref{A31} taking the limit for $m\to\infty$ that
\begin{equation}%\label{A32}
\gathered
\langle\partial_t u(\cdot,t),e_j\rangle=a^2\langle\partial_{zz}u(\cdot,t),e_j\rangle+\langle f_0(u(\cdot,t)+H(\cdot,t)),e_j\rangle+\langle\mu_1,e_j\rangle\\
+\langle B^{\T}(\cdot)x(t),e_j\rangle-\langle H_t(\cdot,t),e_j\rangle
\endgathered
\end{equation}
By the completeness of the set of eigenfunctions  $\{e_j\}_{j=1}^{\infty}$ it follows that for almost all $t\in[0,T]$ we have
\begin{equation}\label{A32}
\partial_t u(z,t)=a^2\partial_{zz}u(z,t)+f_0(u(z,t)+H(z,t))+\mu_1+B^{\T}(z)x(t)-H_t(z,t).
\end{equation}
Hence $\partial_t u(z,t)$ is a sum of functions from $L^{2q/(2q-1)}([0,T],L^{2q/(2q-1)}[0,l])$ and $L^2([0,T],H^{-1}[0,l])$.

We show that $\mu_1=f_1(u(\cdot,t)+H(t))$. Let
\begin{equation}\label{A33}
\varpi_m:=\int\limits_0^{t}\langle f_1(u_m(\cdot,s)+H(\cdot,s))-f_1(v(\cdot,s)+H(\cdot,s)),u_m(\cdot,s)-v(\cdot,s)\rangle\,ds
\end{equation}
By the above considerations $f_1(u_m(\cdot,s)+H(\cdot,s))$, $f_1(v(\cdot,s)+H(\cdot,s))\in L^{2q/(2q-1)}[0,l]$. Applying the integration ba parts as in \eqref{0.2} we get
\begin{equation}\label{A34}
\gathered
\int\limits_0^{t}\langle f_1(u_m(\cdot,s)+H(\cdot,s)),u_m(\cdot,s)\rangle\,ds\\
=\int\limits_0^{t}\langle \partial_tu_m(\cdot,s)-f_0(u_m(\cdot,s)+H(\cdot,s))-a^2\partial_{zz}u_m(\cdot,s)\\
-\Pi_mB^{\T}(z)x_m(s)+\Pi_mH_t(\cdot,s),u_m(\cdot,s)\rangle\,ds\\
=\frac{1}{2}\|u_m(\cdot,t)\|_{L^2[0,l]}^2-\frac{1}{2}\|u_m(\cdot,0)\|_{L^2[0,l]}^2+a^2\int\limits_0^t\|\partial_zu_m(\cdot,s)\|_{L^2[0,l]}^2\,ds\\
-\int\limits_0^t\langle f_0(u_m(\cdot,s)+H(\cdot,s)),u_m(\cdot,s)\rangle\,ds
-\int\limits_0^t\langle B^{\T}(z)x_m(s),u_m(\cdot,s)\rangle\,ds\\
+\int\limits_0^t\langle H_t(\cdot,s),u_m(\cdot,s)\rangle\,ds
\endgathered
\end{equation}
Taking \eqref{0.2} into account we have
\begin{equation}\label{A35}
\gathered
\varpi_m=\frac{1}{2}\|u_m(\cdot,t)\|_{L^2[0,l]}^2-\frac{1}{2}\|u_m(\cdot,0)\|_{L^2[0,l]}^2+a^2\int\limits_0^t\|\partial_zu_m(\cdot,s)\|_{L^2[0,l]}^2\,ds\\
-\int\limits_0^t\langle B^{\T}(z)x_m(s),u_m(\cdot,s)\rangle\,ds+\int\limits_0^t\langle H_t(\cdot,s),u_m(\cdot,s)\rangle\,ds\\
-\int\limits_0^t\langle f_1(u_m(\cdot,s)+H(\cdot,s)),v(\cdot,s)\rangle\,ds\\
-\int\limits_0^t\langle f_1(v(\cdot,s)+H(\cdot,s)),u_m(\cdot,s)-v(\cdot,s)\rangle\,ds\\
-\int\limits_0^t\langle f_0(u_m(\cdot,s)+H(\cdot,s)),u_m(\cdot,s)\rangle\,ds
\endgathered
\end{equation}
Due to the weak convergence $u_m(\cdot,s)\to u(\cdot,s)$ in $L^2[0,l]$ it follows that
\begin{equation}\label{A36}
\lim\inf\limits_{m\to\infty}\|u_m(\cdot,t)\|_{L^2[0,l]}\ge\|u(\cdot,t)\|_{L^2[0,l]}.
\end{equation}
Similarly follows
\begin{equation}\label{A37}
\lim\inf\limits_{m\to\infty}\|\partial_zu_m(\cdot,t)\|_{L^2[0,l]}\ge\|\partial_z u(\cdot,t)\|_{L^2[0,l]}.
\end{equation}
Note that $\|u_m(\cdot,0)\|_{L^2[0,l]}^2\to\|u(\cdot,0)\|_{L^2[0,l]}^2$. From the weak convergence  $u_m\to u$, uniform boundedness in $L^{2}([0,T],L^2[0,l])$ and by the convergence $x_m\to x$ in $C([0,T],\mathbb{R}^n)$  for $m\to\infty$ we get the convergence of
\begin{equation*}
	\gathered
	\int\limits_0^t\langle B^{\T}(z)x_m(s),u_m(\cdot,s)\rangle\,ds\to \int\limits_0^t\langle B^{\T}(z)x(s),u(\cdot,s)\rangle\,ds,\\
	\int\limits_0^t\langle H_t(\cdot,s),u_m(\cdot,s)\rangle\,ds\to \int\limits_0^t\langle H_t(\cdot,s),u(\cdot,s)\rangle\,ds,\\
	\int\limits_0^t\langle f_0(u_m(\cdot,s)+H(\cdot,s)),v(\cdot,s)\rangle\,ds\to \int\limits_0^t\langle \mu,v(\cdot,s)\rangle\,ds,\\
	\int\limits_0^t\langle f_0(v(\cdot,s)+H(\cdot,s)),u_m(\cdot,s)-v(\cdot,s)\rangle\,ds\\\to
	\int\limits_0^t\langle f_0(v(\cdot,s)+H(\cdot,s)),u(\cdot,s)-v(\cdot,s)\rangle\,ds,\\
	\int\limits_0^t\langle f_1(u_m(\cdot,s)+H(\cdot,s)),v(\cdot,s)\rangle\,ds\to \int\limits_0^t\langle \mu,v(\cdot,s)\rangle\,ds,\\
	\int\limits_0^t\langle f_1(v(\cdot,s)+H(\cdot,s)),u_m(\cdot,s)-v(\cdot,s)\rangle\,ds\\\to
	\int\limits_0^t\langle f_1(v(\cdot,s)+H(\cdot,s)),u(\cdot,s)-v(\cdot,s)\rangle\,ds
	\endgathered
\end{equation*}
for $m\to\infty$.
Hence for $m\to\infty$ we have
\begin{equation*}
	\gathered
	\int\limits_0^t\langle f_0(u_m(\cdot,s)+H(\cdot,s)),u_m(\cdot,s)\rangle\,ds
	\to \int\limits_0^t\langle f_0(u(\cdot,s)+H(\cdot,s)),u(\cdot,s)\rangle\,ds
	\endgathered
\end{equation*}
From the other side we have
\begin{equation}\label{A37bis}
\varpi_m=\int\limits_0^{t}\langle f_1(u_m(\cdot,s)+H(\cdot,s))-f_1(v(\cdot,s)+H(\cdot,s)),u_m(\cdot,s)-v(\cdot,s)\rangle\,ds\\
\end{equation}
by the assumptions for $f_1$, in particular that $f_1^{\prime}(s)<0$ for all $s\in\mathbb{R}$, $s\ne 0$ we obtain
\begin{equation}\label{A38}
\gathered
0\ge \varpi_m\ge\frac{1}{2}\|u(\cdot,t)\|_{L^2[0,l]}^2-\frac{1}{2}\|u(\cdot,0)\|_{L^2[0,l]}^2+a^2\int\limits_0^t\|\partial_zu(\cdot,s)\|_{L^2[0,l]}^2\,ds\\
-\int\limits_0^t\langle B^{\T}(z)x(s),u(\cdot,s)\rangle\,ds+\int\limits_0^t\langle H_t(\cdot,s),u(\cdot,s)\rangle\,ds\\
-\int\limits_0^t\langle \mu_1,v(\cdot,s)\rangle\,ds
-\int\limits_0^t\langle f_1(v(\cdot,s)+H(\cdot,s)),u(\cdot,s)-v(\cdot,s)\rangle\,ds\\-
\int\limits_0^t\langle f_0(u(\cdot,s)+H(\cdot,s)),u(\cdot,s)\rangle\,ds
\endgathered
\end{equation}
Using \eqref{0.2} with $u=v$ from \eqref{A32} follows
\begin{equation*}
	\gathered
	\frac{1}{2}\|u(\cdot,t)\|^2_{L^2[0,l]}-\frac{1}{2}\|u(\cdot,0)\|^2_{L^2[0,l]}=-a^2\int\limits_0^t\|\partial_{z}u(\cdot,s)\|_{L^2[0,l]}^2\,ds\\
	+\int\limits_0^t\langle f_0(u(\cdot,s)+H(\cdot,s)),u(\cdot,s)\rangle\,ds+\int\limits_0^t\langle\mu_1,u(\cdot,s)\rangle\,ds\\
	+\int\limits_0^t\langle B^{\T}(z)x(t),u(\cdot,s)\rangle\,ds-\int\limits_0^t\langle H_t(z,t),u(\cdot,s)\rangle\,ds
	\endgathered
\end{equation*}
Comparing this equality with \eqref{A38} we conclude that
\begin{equation}\label{A39}
\gathered
\int\limits_0^t\langle \mu_1-f_1(v(\cdot,s)+H(\cdot,s)),u(\cdot,s)-v(\cdot,s)\rangle\,ds\le 0
\endgathered
\end{equation}
for all $v\in L^{2q}([0,T],L^{2q}[0,l])$. Let $v=u-\nu w$, $\nu\in[0,1]$, for some $\vartheta\in(0,1)$ the following estimation holds
\begin{equation*}
	\gathered
	|f_1(u+\nu w+H)-f_1(u+H)|
	\le |f_1^{\prime}(u+\nu\vartheta w+H)|\nu|w|\\
	\le c_0(1+|u+\nu\vartheta w+H|^{2q-2})\nu|w|\\
	\le c_0(1+3^{2q-3}(|u|^{2q-2}+|w|^{2q-2}+|H|^{2q-2}))\nu|w|\\
	\le \nu(c_0+3^{3q-2}c_0\sup_{(z,s)\in[0,l]\times[0,T]}|H(z,s)|^{2q-2})|w|\\+
	3^{2q-3}c_0\nu|u|^{2q-2}|w|+3^{2q-3}c_0\nu|w|^{2q-1}\\:
	=\nu(\kappa_1|w|+\kappa_2|u|^{2q-2}|w|+\kappa_3|w|^{2q-1}).
	\endgathered
\end{equation*}

Let $w\in L^{2q}([0,T],L^{2q}[0,l])$. Taking\eqref{A30bis} and \eqref{A30*} into account we obtain
\begin{equation*}
	\gathered
	\int\limits_0^t\langle f_1(u(\cdot,s)+\nu w(\cdot,s)+H(\cdot,s))-f_1(u(\cdot,s)+H(\cdot,s)),w(\cdot,s)\rangle\,ds\\
	=\nu\int\limits_0^t\int\limits_0^l(\kappa_1|w(z,s)|^2+\kappa_2|u(z,s)|^{2q-2}|w(z,s)|^2+\kappa_3|w(z,s)|^{2q})\,dz\,ds\\
	\le\nu(\kappa_1\|w\|_{L^2([0,T],L^2[0,l])}^2+(\kappa_3+\frac{\kappa_2}{q})\|w\|_{L^{2q}([0,T],L^{2q}[0,l])}^{2q}\\
	+\frac{\kappa_2(q-1)}{q}\|u\|_{L^{2q}([0,T],L^{2q}[0,l])}^{2q})\to 0\quad\text{for}\quad \nu\to 0.
	\endgathered
\end{equation*}
We substitute $v=u-\nu w$, $\nu\in(0,1]$ into \eqref{A39} and get
\begin{equation}\label{A40}
\int\limits_0^t\langle \mu_1-f_1(u(\cdot,s)+H(\cdot,s)-\nu w(\cdot,s)),w(\cdot,s)\rangle\,ds\le 0.
\end{equation}
Taking the limit for $\nu\to 0$ it follows that
\begin{equation*}
	\int\limits_0^t\langle \mu_1-f_1(u(\cdot,s)+H(\cdot,s)),w(\cdot,s)\rangle\,ds\le 0.
\end{equation*}
Since $w$ is arbitrary it follows that
\begin{equation*}
	\int\limits_0^t\langle \mu_1-f_1(u(\cdot,s)+H(\cdot,s)),w(\cdot,s)\rangle\,ds= 0.
\end{equation*}
and $\mu_1=f_1(u(\cdot,t)+H(\cdot,t)$.
From Lemma \ref{lemmaA1} and the equality
\begin{equation*}
	\gathered
	\dot x_m(t)=Cx_m(t)+X(x_m(t))+\int\limits_0^lD(z)u_m(z,t)\,dz+\int\limits_0^lD(z)H(z,t)\,dz
	\endgathered
\end{equation*}
we conclude the uniform convergence of  $\{\dot x_m(t)\}_{m=m_0}^{\infty}$ on $[0,T]$. Hence in the last equality we can take the limit for $m\to\infty$ which finishes the proof of Theorem \ref{main-theorem}.
\end{proof}
\section{Calculation of $\dot V$ in the proof of Theorem \ref{th2}}\label{detailed-calculations}
Here we calculate the full time derivative of the Lyapunov functional
\begin{equation*}
	\gathered
	V(v(\cdot,t),x)=\|v(\cdot,t)\|^2_{L^2[0,l]}+2x^{\T}\int\limits_0^{l}P_{12}(z)v(z,t)\,dz+x^{\T}Px
	\endgathered
\end{equation*}
as follows. We note that by Lemma 2.2 from \cite{DKS20} we have
\begin{equation*}
	\frac{d}{dt}\|v(\cdot,t)\|^2_{L^2[0,l]}=2\langle\partial_t v,v\rangle,
\end{equation*}
and similarly
\begin{equation*}
	\gathered
	\frac{d}{dt}\int\limits_0^{l}P_{12}(z)v(z,t)\,dz=2\langle\partial_t v,P_{12}\rangle,
	\endgathered
\end{equation*}
so that
\begin{equation*}
	\gathered
	\dot V(v(\cdot,t),x)=2\langle\partial_t v(\cdot,t),v(\cdot,t)\rangle\\+2
	\int\limits_0^l(Cx+X(x)+\int\limits_0^lD(\xi)v(\xi,t)\,d\xi+p(t))^{\T}P_{12}(z)v(z,t)\,dz\\
	+2\langle\partial_t v(\cdot,t), x^{\T}P_{12}\rangle+2x^{\T}P(Cx+X(x)+\int\limits_0^l D(z)v(z,t)\,dz+p(t))\\
	=2a^2\langle\partial_{zz}v,v(\cdot,t)\rangle+2\int\limits_0^l f(v(z,t))v(z,t)\,dz+
	2\int\limits_0^l B^{\T}(z)x v(z,t)\,dz\\
	+2\int\limits_0^l g(z,t)v(z,t)\,dz+2x^{\T}\int\limits_0^lC^{\T}P_{12}(z)v(z,t)\,dz+
	2X^{\T}(x)\int\limits_0^lP_{12}(z)v(z,t)\,dz\\
	+2\int\limits_0^l\Big(\int\limits_0^lv(\xi,t)D^{\T}(\xi)\,d\xi\Big)P_{12}(z)v(z,t)\,dz\\+
	2p^{\T}(t)\int\limits_0^lP_{12}(z)v(z,t)\,dz+2x^{\T}\langle a^2\partial_{zz}v,P_{12}\rangle\\
	+2x^{\T}\int\limits_0^lP_{12}(z)(f(v(z,t))+B^{\T}(z)x+g(z,t))\,dz\\
	+x^{\T}(C^{\T}P+PC)x+2x^{\T}PX(x)+2x^{\T}P\int\limits_0^lD(z)v(z,t)\,dz
	+2x^{\T}Pp(t)
	\endgathered
\end{equation*}
By the definition of the generalized function derivative and taking the boundary conditions
 $P_{12}(0)=P_{12}(l)=0$ into account we have
\begin{equation*}
	\gathered
	\langle\partial_{zz}v,v(\cdot,t)\rangle=
	-\langle\partial_{z}v,\partial_{z}v\rangle=-\|\partial_z v(\cdot,t)\|^2_{L^2[0,l]},\\
	\langle\partial_{zz}v,P_{12}\rangle=\langle v,\partial_{zz}P_{12}\rangle=
	\int\limits_0^lv(z,t)\partial_{zz}P_{12}(z)\,dz.
	\endgathered
\end{equation*}
Hence
\begin{equation*}
	\gathered
	\dot V(v(\cdot,t),x)=-2\|\partial_z v\|^2_{L^2[0,l]}+2\int\limits_0^l f(v(z,t))v(z,t)\,dz\\+
	2\int\limits_0^lP_{12}(z)v(z,t)\,dz\int\limits_0^lv(\xi,t)D^{\T}(\xi)\,d\xi\\
	+x^{\T}(C^{T}P+PC)x+2x^{\T}\int\limits_0^lP_{12}(z)B^{\T}(z)\,dzx+
	2x^{\T}PX(x)\\
	+2\int\limits_0^l g(z,t)v(z,t)\,dz\\ 
	+2x^{\T}Pp(t)+2x^{\T}\int\limits_0^lP_{12}(z)g(z,t)\,dz
	+2p^{\T}(t)\int\limits_0^lP_{12}(z)v(z,t)\,dz\\
	+2X^{\T}(x)\int\limits_0^lP_{12}(z)v(z,t)\,dz+
	2x^{\T}\int\limits_0^lP_{12}(z)f(v(z,t))\,dz
	+2\int\limits_0^l B^{\T}(z)x v(z,t)\,dz\\
	+2x^{\T}\int\limits_0^lC^{\T}P_{12}(z)v(z,t)\,dz\\
	+2x^{\T}\int\limits_0^lv(z,t)a^2\partial_{zz}P_{12}(z)\,dz+
	2x^{\T}P\int\limits_0^lD(z)v(z,t)\,dz
	\endgathered
\end{equation*}
\begin{equation*}
	\gathered
	=W_1+W_2+W_3+W_4\\
	+2x^{\T}\int\limits_0^l(a^2 P_{12}^{\prime\prime}(z)+C^{\T}P_{12}(z)+B(z)+PD(z))v(z,t)\,dz
	\endgathered
\end{equation*}
By the choice of $P_{12}(z)$ the last integral vanishes (see \eqref{11bis}).
\section{Proof of Corollary \ref{unique}}
\begin{proof}
	Assume that $(u_i,x_i)\in L^2(\mathbb{R}_+,L^{\infty}[0,l])\cap L^{2q}(\mathbb{R}_+,L^{2q}[0,l])\in C(\mathbb{R}_+,\mathbb{R}^n)$, $i=1,2$ satisfy \eqref{1} -- \eqref{3}. Then $w:=u_2-u_1$, $x:=x_2-x_1$ solve the following problem
	\begin{equation*}
	\gathered
	w_t(z,t)=a^2w_{zz}(z,t)+f(w(z,t)+u_1(z,t))-f(u_1(z,t))+B^{\T}(z)x(t),\\
	\dot x(t)=Cx(t)+X(x(t)+x_1(t))-X(x_1(t))+\int\limits_0^lD(z)w(z,t)\,dz,
	\endgathered
	\end{equation*}
	subject to initial conditions
	\begin{equation*}
	x(0)=0\in\mathbb{R}^n,\quad w(z,0)=0,\quad z\in(0,l),\quad t\in(0,+\infty)
	\end{equation*}
	and boundary conditions
	\begin{equation*}
	w(0,t)=0,\quad w(l,t)=0.
	\end{equation*}
	From Theorem \ref{th2} follows $\sup_{t\in\mathbb{R}_+}|x_1(t)|<\infty$, $\sup_{t\in\mathbb{R}_+}\|u_1(\cdot,t)\|_{L^{\infty}[0,l]}<\infty$,
	hence for some constant $L_1>0$ we have
	\begin{equation*}
	\|X(x(t)+x_1(t))-X(x_1(t))\|\le L_1\|x(t)\|.
	\end{equation*}
	We introduce a vector Lyapunov function with components defined by  $$V_1(w(\cdot,t))=\frac{1}{2}\|w(\cdot,t)\|_{L^2[0,l]}^2\qquad V_2(x(t))=\frac{1}{2}\|x(t)\|^2$$
	so that
	\begin{equation*}
	\gathered
	\dot V_1(w(\cdot,t))\le-\frac{2\pi^2a^2}{l^2}V_1(w(\cdot,t))+\int\limits_0^lf^{\prime}(u_1(z,t)+\vartheta w(z,t))w^2(z,t)\,dz\\
	+\|B\|_{L^2[0,l]}\|x(t)\|\|w(\cdot,t)\|_{L^2[0,l]}\\
	\le
	(-\frac{2\pi^2a^2}{l^2}+\|B\|_{L^2[0,l]})V_1(w(\cdot,t))
	+\|B\|_{L^2[0,l]}V_2(x(t)),\\
	\dot V_2(x(t))\le \|C\|V_2(x(t))+2L_2V_2(x(t))+\|D\|_{L^2[0,l]}\|x(t)\|\|w(\cdot,t)\|_{L^2[0,l]}\\
	\le (2\|C\|+2L_2+\|D\|_{L^2[0,l]})V_2(x(t))+\|D\|_{L^2[0,l]}V_1(w(\cdot,t)),
	\endgathered
	\end{equation*}
	for some $\vartheta\in(0,1)$, which implies that for all $t\ge 0$
	\begin{equation*}
	\gathered
	0\le \begin{pmatrix}
	V_1(w(\cdot,t))\\
	V_2(x(t))
	\end{pmatrix}
	\le e^{M t}
	\begin{pmatrix}
	V_1(w(\cdot,0))\\
	V_2(x(0))
	\end{pmatrix}=0,
	\endgathered
	\end{equation*}
	with some constant matrix $M$. Hence $w(\cdot,t)=0$, $x(t)= 0$ for a.a. $t$, which proves the lemma.
\end{proof}

\bibliography{references}

\begin{thebibliography}{10}

\bibitem{ahmed-ali}
T.~Ahmed-Ali, I.~Karafyllis, F.~Giri, M.~Krstic, and F.~Lamnabhi-Lagarrigue.
\newblock Exponential stability analysis of sampled-data {ODE}-{PDE} systems
  and application to observer design.
\newblock {\em IEEE Trans. Automat. Control}, 62(6):3091--3098, 2017.

\bibitem{tanwani}
S.~T. Aneel~Tanwani, Christophe~Prieur.
\newblock Disturbance-to-state stabilization and quantized control for linear
  hyperbolic systems.
\newblock {\em arXiv.org}, (1703.00302v2.[csSY]):1--22, 2017.

\bibitem{babin-v}
A.~V. Babin and M.~I. Vishik.
\newblock {\em Attractors of evolution equations}, volume~25 of {\em Studies in
  Mathematics and its Applications}.
\newblock North-Holland Publishing Co., Amsterdam, 1992.
\newblock Translated and revised from the 1989 Russian original by Babin.

\bibitem{bartyshev}
A.~Bartyshev.
\newblock Application of ljapunov vector-valued functions for investigating
  two-component systems. in vector-valued ljapunov functions and their
  construction (russian).
\newblock {\em Nauka}, pages 237--257,288, 1980.

\bibitem{KrD74}
J.~L. Dalecki\u{\i} and M.~G. Kre\u{\i}n.
\newblock {\em Stability of solutions of differential equations in {B}anach
  space}.
\newblock American Mathematical Society, Providence, R.I., 1974.
\newblock Translated from the Russian by S. Smith, Translations of Mathematical
  Monographs, Vol. 43.

\bibitem{DKS20}
S.~Dashkovskiy, O.~Kapustyan, and J.~Schmid.
\newblock A local input-to-state stability result w.r.t. attractors of
  nonlinear reaction-diffusion equations.
\newblock {\em Math. Control Signals Systems}, 32(3):309--326, 2020.

\bibitem{DaN07}
S.~Dashkovskiy and A.~Narimanyan.
\newblock Thermal plasma cutting. {I}. {M}odified mathematical model.
\newblock {\em Math. Model. Anal.}, 12(4):441--458, 2007.

\bibitem{Dat72}
R.~Datko.
\newblock Uniform asymptotic stability of evolutionary processes in a {B}anach
  space.
\newblock {\em SIAM J. Math. Anal.}, 3:428--445, 1972.

\bibitem{efendiev}
M.~A. Efendiev, M.~Otani, and H.~J. Eberl.
\newblock Mathematical analysis of a {PDE}-{ODE} coupled model of mitochondrial
  swelling with degenerate calcium ion diffusion.
\newblock {\em SIAM J. Math. Anal.}, 52(1):543--569, 2020.

\bibitem{GGT20}
M.~Garavello, P.~Goatin, T.~Liard, and B.~Piccoli.
\newblock A multiscale model for traffic regulation via autonomous vehicles.
\newblock {\em J. Differential Equations}, 269(7):6088--6124, 2020.

\bibitem{Hen81}
D.~Henry.
\newblock {\em Geometric theory of semilinear parabolic equations}, volume 840
  of {\em Lecture Notes in Mathematics}.
\newblock Springer-Verlag, Berlin-New York, 1981.

\bibitem{JSZ19}
B.~Jacob, F.~L. Schwenninger, and H.~Zwart.
\newblock On continuity of solutions for parabolic control systems and
  input-to-state stability.
\newblock {\em J. Differential Equations}, 266(10):6284--6306, 2019.

\bibitem{kang}
W.~Kang and E.~Fridman.
\newblock Boundary control of cascaded ode-heat equations under actuator
  saturation.
\newblock {\em arXiv:1608.03729v1 [mathOC],19 Apr.2017}.

\bibitem{kang2}
W.~Kang and E.~Fridman.
\newblock Boundary control of delayed {ODE}-heat cascade under actuator
  saturation.
\newblock {\em Automatica J. IFAC}, 83:252--261, 2017.

\bibitem{KaK19}
I.~Karafyllis and M.~Krstic.
\newblock {\em Input-to-state stability for {PDE}s}.
\newblock Communications and Control Engineering Series. Springer, Cham, 2019.

\bibitem{kedyk}
T.~V. Kedyk.
\newblock Invariant manifolds of hybrid quasimonotone extensions.
\newblock {\em Dokl. Akad. Nauk Ukrain. SSR}, (10):8--11, 179, 1991.

\bibitem{kedyk-ob}
T.~V. Kedyk and A.~Y. Obolenski\u{\i}.
\newblock On the stability with respect to two measures of hybrid quasimonotone
  extensions.
\newblock {\em Dokl. Akad. Nauk Ukrain. SSR}, (8):80--82, 1991.

\bibitem{krst}
M.~Krstic.
\newblock Compensating actuator and sensor dynamics governed by diffusion pdes.
\newblock {\em Syst. Control Lett.}, 58:372--377, 2009.

\bibitem{lions}
J.-L. Lions.
\newblock {\em Quelques m\'{e}thodes de r\'{e}solution des probl\`emes aux
  limites non lin\'{e}aires}.
\newblock Dunod; Gauthier-Villars, Paris, 1969.

\bibitem{hirsh}
H.~S. M.~Hirsch.
\newblock Monotone dynamical systems.
\newblock {\em ser. Handbook of Differential Equations: Ordinary Differential
  Equations}, 2:239--357, 2006.

\bibitem{mart-sl}
A.~A. Martynyuk and V.~I. Slynko.
\newblock On the stability of linear hybrid mechanical systems with a
  distributed component.
\newblock {\em Ukra\"{\i}n. Mat. Zh.}, 60(2):204--216, 2008.

\bibitem{Mik78}
V.~P. Mikha\u{\i}lov.
\newblock {\em Partial differential equations}.
\newblock ``Mir'', Moscow; distributed by Imported Publications, Inc., Chicago,
  Ill., 1978.
\newblock Translated from the Russian by P. C. Sinha.

\bibitem{MiP20}
A.~Mironchenko and C.~Prieur.
\newblock Input-to-{S}tate {S}tability of {I}nfinite-{D}imensional {S}ystems:
  {R}ecent {R}esults and {O}pen {Q}uestions.
\newblock {\em SIAM Rev.}, 62(3):529--614, 2020.

\bibitem{miron}
A.~Mironchenko, C.~Prieur, and F.~Wirth.
\newblock Desing of saturated controls for an unstable parabolic {PDE}.
\newblock {\em Joint 8th IFAC Symposium on Mechatronic Systems (MECHATRONICS
  19) and 11th IFAC Symposium on Nonlinear Control Systems (NOLCOS 19)}, Sep
  2019, Vienne, Austria.

\bibitem{MiW19}
A.~Mironchenko and F.~Wirth.
\newblock Non-coercive {L}yapunov functions for infinite-dimensional systems.
\newblock {\em J. Differential Equations}, 266(11):7038--7072, 2019.

\bibitem{laurent}
P.~G. Nicolas Laurent-Brouty, Guillaume~Costeseque.
\newblock A coupled pde-ode model for bounded acceleration in macroscopic
  traffic flow models.
\newblock {\em IFAC-PapersOnLine, Elsevier}, 51(9):37--42, 2018.

\bibitem{salvidar}
B.~Saldivar, S.~Mondi\'{e}, and J.~C. \'{A}vila Vilchis.
\newblock The control of drilling vibrations: a coupled {PDE}-{ODE} modeling
  approach.
\newblock {\em Int. J. Appl. Math. Comput. Sci.}, 26(2):335--349, 2016.

\bibitem{Vla71}
V.~S. Vladimirov.
\newblock {\em Equations of mathematical physics}, volume~3 of {\em Translated
  from the Russian by Audrey Littlewood. Edited by Alan Jeffrey. Pure and
  Applied Mathematics}.
\newblock Marcel Dekker, Inc., New York, 1971.

\bibitem{Xu}
Z.~Xu, Y.~Liu, and J.~Li.
\newblock Adaptive stabilization for a class of {PDE}-{ODE} cascade systems
  with uncertain harmonic disturbances.
\newblock {\em ESAIM Control Optim. Calc. Var.}, 23(2):497--515, 2017.

\bibitem{yuan}
Y.~Yuan, Z.~Shen, and F.~Liao.
\newblock Stabilization of coupled {ODE}-{PDE} system with intermediate point
  and spatially varying effects interconnection.
\newblock {\em Asian J. Control}, 19(3):1060--1074, 2017.

\end{thebibliography}
\bibliographystyle{abbrv}

\end{document}